\documentclass[a4paper,10pt,leqno]{amsart}
\usepackage[total={6in,9in},
top=1in, left=1in, right=1in, bottom=1in]{geometry}
\usepackage{amsmath}
\usepackage[italian, english]{babel}
\usepackage{amsfonts}
\usepackage{amssymb}
\usepackage{dsfont}
\usepackage{mathrsfs}
\usepackage{graphicx}
\usepackage{setspace}
\usepackage{fancyhdr}
\usepackage{amsthm}
\usepackage{empheq}
\usepackage{cases}
\usepackage[all]{xy}
\usepackage{stmaryrd}
\newcommand{\erre}{\mathds{R}}

\newcommand{\enne}{\mathds{N}}

\newcommand{\diver}{\operatorname{div}}

\newcommand{\ra}{\rightarrow}

\newcommand{\set}[1]{{\left\{#1\right\}}}               
\newcommand{\pa}[1]{{\left(#1\right)}}                  
\newcommand{\sq}[1]{{\left[#1\right]}}                  
\newcommand{\abs}[1]{{\left|#1\right|}}                 
\newcommand{\pair}[1]{\left\langle#1\right\rangle}      

\newcommand{\eps}{\varepsilon}                           

\renewcommand{\tilde}[1]{\widetilde{#1}}

\newtheorem{theorem}{\textbf{Theorem}}[section]
\newtheorem{lemma}[theorem]{\textbf{Lemma}}
\newtheorem{proposition}[theorem]{\textbf{Proposition}}
\newtheorem{cor}[theorem]{\textbf{Corollary}}
\newtheorem{defi}[theorem]{\textbf{Definition}}
\theoremstyle{remark}
\newtheorem{rem}[theorem]{\textbf{Remark}}

\newtheorem{exe}[theorem]{\textbf{Example}}
\numberwithin{equation}{section}

\title[]
{Nonexistence results for elliptic differential inequalities with
a potential on Riemannian manifolds}

\date{\today} 

\keywords{}

\subjclass[2010]{35B09; 35B53; 35J62; 58J05   }

\begin{document}

\maketitle

\begin{center}
\textsc{\textmd{P. Mastrolia\footnote{Universit\`{a} degli Studi
di Milano, Italy. Email: paolo.mastrolia@unimi.it. Partially
supported by FSE, Regione Lombardia.}, D. D.
Monticelli\footnote{Universit\`{a} degli Studi di Milano, Italy.
Email: dario.monticelli@unimi.it.} and F.
Punzo\footnote{Universit\`{a} degli Studi di Milano, Italy. Email:
fabio.punzo@unimi.it. \\ The three authors are supported by GNAMPA
project ``Analisi globale ed operatori degeneri'' and are members
of the Gruppo Nazionale per l'Analisi Matematica, la
Probabilit\`{a} e le loro Applicazioni (GNAMPA) of the Istituto
Nazionale di Alta Matematica (INdAM). }, }}
\end{center}
\begin{abstract}
In this paper we are concerned with a class of elliptic
differential inequalities with a potential both on $\erre^m$ and
on Riemannian manifolds. In particular, we investigate the effect
of the geometry of the underlying manifold and of the behavior of
the potential at infinity on nonexistence of nonnegative
solutions.
\end{abstract}


\section{Introduction}

One of the most important and well-studied class of elliptic
differential inequalities in Global Analysis, due to its
ubiquitous presence in many applications, is
\begin{equation}\label{EQ_lapl}
 \Delta u+V(x)u^\sigma \leq 0,
\end{equation}
both on $\erre^m$ and on general Riemannian manifolds $(M,g)$,
where $\Delta$ denotes the Laplace-Beltrami operator associated to
the metric and $\sigma>1$. In particular, in many instances it is
also required that the solution $u$ of the problem is positive.

The aim of this paper is to investigate in depth the influence of
the geometry of the underlying complete, noncompact Riemannian
manifold $(M,g)$ of dimension $m$ and of the potential $V$ on the
existence of positive solutions to the class of elliptic
differential inequalities
\begin{equation}\label{EQ_gen}
 \frac{1}{a(x)}\diver\pa{a(x)\abs{\nabla u}^{p-2}\nabla u}+V(x)u^\sigma
 \leq0,
 \end{equation}
thus highlighting the interplay between analysis and geometry in
this class of problems, which includes \eqref{EQ_lapl}. Here and
in the rest of the paper we assume that $a : M \ra \erre$
satisfies
\begin{equation}\label{EQ_propr_a}
  a>0, \quad a \in \operatorname{Lip}_{\text{loc}}(M),
\end{equation}
$V>0$ a.e. on $M$ and $V\in L^1_{\text{loc}}(M)$, and the
constants $p$ and $\sigma$ satisfy $p>1$, $\sigma>p-1$. In our
results, the geometry of $M$ appears through conditions on the
growth of suitably weighted volumes of geodesic balls, involving
both the potential $V$ and the function $a$. We explicitly note
that some of the results we find are new also in the specific case
of the model equation \eqref{EQ_lapl}.

This class of problems has a very long history, particularly in
the Euclidean setting, starting with the seminal works of Gidas
\cite{Gidas} and Gidas-Spruck \cite{GidasSpruck}. In those papers
the authors show, among other results, that any nonnegative
solution of equation \eqref{EQ_lapl} is in fact identically null
if and only if $\sigma\leq \frac{m}{m-2}$, in case $V\equiv 1$ and
$m\geq 3$.

We refer to the interesting papers of Mitidieri-Pohozaev
\cite{MitPohoz359}, \cite{MitPohoz227}, \cite{MitPohozAbsence},
\cite{MitPohozMilan} for a comprehensive description of results
related to these (and also more general) problems on $\erre^m$.
Note that analogous results have also been obtained for degenerate
elliptic equations and inequalities (see, e.g., \cite{DaLu},
\cite{Mont}), and for the parabolic companion problems (see, e.g.,
\cite{MitPohozMilan}, \cite{PoTe}, \cite{P1}, \cite{PuTe}). Using
nonlinear capacity arguments, which exploit suitably chosen test
functions, Mitidieri-Pohozaev prove nonexistence of weak or
distributional solutions  of wide classes of differential
inequalities in $\erre^m$, which include also many of the examples
that we consider here. In particular, they show that equation
\eqref{EQ_lapl} on $\erre^m$ does not admit any nontrivial
nonnegative solution, provided that
\begin{equation}\label{28}
\liminf_{R\rightarrow+\infty}R^{-\frac{2\sigma}{\sigma-1}}\int_{B_{\sqrt{2}R}\setminus
B_R}V^{-\frac{1}{\sigma-1}}dx<\infty.
\end{equation}
For the case of equation \eqref{EQ_gen} on $\erre^m$, the authors
show that no nonnegative, nontrivial solution exists in case
$V\equiv1$, $m>p$ and $p-1<\sigma\leq\frac{m(p-1)}{m-p}$. This can
again be read as a condition relating the volume growth of
Euclidean balls, which depends on the dimension $m$ of the space,
and the exponent of the nonlinearity in the equation.


The results in the case of a complete Riemannian manifold have a
more recent history, in particular we cite the inspiring papers of
Grygor'yan-Kondratiev \cite{GrigKond} and Grygor'yan-Sun
\cite{GrigS}, whose approach originates from the  work of Kurta
\cite{Kurta}. Using a capacity argument which only exploits the
gradient of the distance function from a fixed reference point,
in particular the authors showed in \cite{GrigKond} that equation
\eqref{EQ_gen}, in case $p=2$,  admits a unique nonnegative weak
solution provided that there exist positive constants $C$, $C_0$
such that for every $R>0$ sufficiently large and every small enough
$\eps>0$,
\begin{equation}\label{Eq_noimp}
\int_{B_R}a V^{-\beta+\eps}\,d\mu_0 \leq C
R^{\alpha+C_0\eps}\pa{\log R}^k,
\end{equation}
where $d\mu_0$ is the canonical Riemannian measure on $M$, $B_R$ is
the geodesic ball centered at a point $x_0\in M$ and
\[
\alpha = \frac{2\sigma}{\sigma-1}, \quad \beta = \frac{1}{\sigma-1}, \quad 0 \leq k <\beta.
\]
Let $r(x)$ denote the geodesic distance of a point $x\in M$ from a
fixed origin $o\in M$. Note that condition \eqref{Eq_noimp} is
satisfied if, for instance, for every $R>0$ large enough one has
\begin{equation*}
    V(x) \leq C \pa{1+r(x)}^{C_0}
    \end{equation*}
   and
  \begin{equation*}
  \int_{B_R} V^{-\beta}\,d\mu \leq C R^{\alpha}\pa{\log R}^k
   \end{equation*}
for some positive constants $C$, $C_0$. The sharpness of the
exponent $\alpha$ in this type of results is evident from the
Euclidean case \eqref{EQ_lapl} with $V\equiv 1$ and $m\geq 3$,
where $\alpha=m$ and the corresponding critical growth is $\sigma
=\frac{m}{m-2}$. The sharpness of the exponent $\beta$ is
definitely a more delicate
 question and has recently been settled on a general Riemannian
manifold, in case $a\equiv1$ and $V\equiv 1$, in \cite{GrigS}. In
that paper, using carefully chosen families of test functions, the
authors showed that equation \eqref{EQ_lapl} with $V\equiv 1$ does
not admit any nonnegative weak solution provided \eqref{Eq_noimp}
holds with $k=\beta$.

In this work we intend to further focus our attention on these
classes of differential inequalities, with the objective of adding
some new results to the already very interesting overall picture.
Our results concerning equation \eqref{EQ_lapl}, in their simplest
form, are contained in the two following theorems.

\begin{theorem}\label{thm_intro1}
Let $(M,g)$ be a complete Riemannian manifold, $\sigma>1$, $V\in
L^1_\text{loc}(M)$ with $V>0$ a.e. on $M$. Define
\begin{equation}\label{30}
\alpha = \frac{2\sigma}{\sigma-1}, \quad \beta = \frac{1}{\sigma-1}
\end{equation}
and assume that there exist $C$, $C_0>0$ such that for every $R>0$
sufficiently large one has
  \begin{equation}\label{27b}
  \int_{B_R\setminus B_{R/2}} V^{-\beta}\,d\mu \leq C R^{\alpha}\pa{\log R}^\beta
   \end{equation}
   and
\begin{equation}\label{27a}
    \frac{1}{C}\pa{1+r(x)}^{-C_0} \leq V(x) \leq C
    \pa{1+r(x)}^{C_0}\qquad\text{ a.e. on }M.
    \end{equation}
Then the only nonnegative weak solution of \eqref{EQ_lapl} is
$u\equiv0$.
\end{theorem}

\begin{theorem}\label{thm_intro2}
With the same notation of Theorem \ref{thm_intro1}, assume that
there exist $C_0\geq0$, $k\geq 0$, $\theta>0$,
$\tau>\max\{\frac{\sigma-1}{\sigma}\pa{k+1},1\}$  such that for
every sufficiently large $R>0$
    \begin{equation}\label{29b}
     \int_{B_R\setminus B_{R/2}} V^{-\beta}\,d\mu \leq C R^{\alpha}\pa{\log R}^k
    \end{equation}
and
    \begin{equation}\label{29a}
    V(x) \leq C \pa{1+r(x)}^{C_0}e^{-\theta \pa{\log
    r(x)}^\tau}\qquad\text{ a.e. on }M.
    \end{equation}
Then the only nonnegative weak solution of \eqref{EQ_lapl} is
$u\equiv0$.
\end{theorem}

A few remarks are now in order. For equation \eqref{EQ_lapl} on
$\erre^m$, condition \eqref{27a} is not required in the results of
Mitidieri-Pohozaev. On the other hand, the potential
$V(x)=\pa{\log(2+|x|^2)}^{-1}$ and the choice of the exponent
$\sigma=\frac{m}{m-2}$ with $m\geq 3$ satisfy the conditions of
our Theorem \ref{thm_intro1}, but they do not satisfy condition
\eqref{28}.

Moreover,  Theorem \ref{thm_intro1} and Theorem \ref{thm_intro2}
extend the result contained in \cite{GrigS} for problem
\eqref{EQ_lapl} on a complete Riemannian manifold $(M,g)$, where
only the case of a constant potential is considered, to the case of
a nonconstant $V$.

On the other hand, while (as we have already noted) the exponent
$\alpha=\frac{2\sigma}{\sigma-2}$ in the power of $R$ in
conditions \eqref{27b} and \eqref{29b} is indeed sharp, Theorem
\ref{thm_intro2} also shows that the sharpness of the exponent of
the term $\log R$ for this type of results is a notion which is
also related to the behavior of the potential $V$ at infinity. In
particular, if $V$ decays at infinity faster than any power of
$r(x)$, as in condition \eqref{29a}, then the critical threshold
for the power of the logarithmic term in the volume growth
condition \eqref{29b} for the nonexistence of nonnegative,
nontrivial solutions of \eqref{EQ_lapl} correspondingly increases
to $\frac{\sigma\tau}{\sigma-1}-1>\beta=\frac{1}{\sigma-1}$. We
explicitly  note that the type of phenomenon described in Theorem
\ref{thm_intro2} has not been pointed out before in literature, to
the best of our knowledge.

The aforementioned theorems are a consequence of more general
results concerning nonnegative weak solutions of inequality
\eqref{EQ_gen}, showing that similar phenomena also occur for this
larger class of problems, which includes the case of inequalities
involving the $p$-Laplace operator. We start with a definition
describing the weighted volume growth conditions on geodesic balls
of $(M,g)$, that will be used in obtaining the nonexistence
results for nonnegative solutions of \eqref{EQ_gen}. We recall
that with $d\mu_0$ we denote the canonical Riemannian measure on
$M$, while we define
\begin{equation}\label{meas_a}
d\mu = a\, d\mu_0
\end{equation}
the weighted measure on $M$ with density $a$.

\begin{defi}\label{defi_volgrowth}
  Let $p>1$, $\sigma>p-1$, $V>0$ a.e. on $M$ and $V\in
  L^1_{\text{loc}}(M)$. Define
  \begin{equation}\label{31}
  \alpha=\frac{p\sigma}{\sigma-p+1},  \qquad \beta=\frac{p-1}{\sigma-p+1}.
  \end{equation}
  We introduce the following three weighted volume growth conditions:
\begin{itemize}
    \item[i)] We say that \emph{condition (HP1)} holds if there exist $C_0>0$, $k\in [0, \beta)$ such that, for every $R>0$  sufficiently large and every $\eps>0$ sufficiently small,
    \begin{equation}\label{EQ_hp1}
      \int_{B_{R}\setminus B_{R/2}} V^{-\beta+\eps}\,d\mu \leq C R^{\alpha+C_0\eps}\pa{\log R}^k.
    \end{equation}
    \item[ii)] We say that \emph{condition (HP2)} holds if  there exists $C_0>0$ such that, for every $R>0$  sufficiently large and every $\eps>0$ sufficiently small,
    \begin{equation}\label{EQ_hp2}
      \int_{B_{R}\setminus B_{R/2}} V^{-\beta+\eps}\,d\mu \leq C R^{\alpha+C_0\eps}(\log R)^\beta \qquad \text{ and } \qquad \int_{B_{R}\setminus B_{R/2}} V^{-\beta-\eps}\,d\mu \leq C R^{\alpha+C_0\eps}(\log R)^\beta.
    \end{equation}
       \item[iii)] We say that \emph{condition (HP3)} holds if there exist $C_0\geq0$, $k\geq 0$, $\theta>0$, $\tau>\max\{\frac{\sigma-p+1}{\sigma}\pa{k+1},1\}$  such that, for every  sufficiently large $R>0$  and every $\eps>0$ sufficiently small,
    \begin{equation}\label{EQ_hp3}
      \int_{B_{R}\setminus B_{R/2}} V^{-\beta+\eps}\,d\mu \leq C R^{\alpha+C_0\eps}\pa{\log R}^k e^{-\eps \theta \pa{\log
      R}^\tau}.
    \end{equation}
  \end{itemize}
\end{defi}
We explicitly note that, when $p=2$, the definitions of $\alpha$,
$\beta$ given in \eqref{30} and \eqref{31} agree.
\begin{rem}
The following are sufficient conditions that imply the above
weighted volume growth conditions for geodesic balls in $M$.
  \begin{itemize}
    \item[i)] Suppose that there exist $C_0>0$, $k>0$ such that
    \begin{equation}\label{e30}
    V(x) \leq C \pa{1+r(x)}^{C_0}
    \end{equation}
    and
    \begin{equation}\label{e31}
    \int_{B_R\setminus B_{R/2}} V^{-\beta}\,d\mu \leq C R^{\alpha}\pa{\log R}^k
    \end{equation}
    for every $R>0$  sufficiently large; then condition \eqref{EQ_hp1} holds.
    \item[ii)] Suppose that there exists $C_0>0$ such that
    \begin{equation}\label{e32}
    \frac{1}{C}\pa{1+r(x)}^{-C_0} \leq V(x) \leq C \pa{1+r(x)}^{C_0}
    \end{equation}
   and
  \begin{equation}\label{e33}
  \int_{B_R\setminus B_{R/2}} V^{-\beta}\,d\mu \leq C R^{\alpha}\pa{\log R}^\beta
   \end{equation}
  for every $R>0$  sufficiently large; then condition \eqref{EQ_hp2} holds.
    \item[iii)] Suppose there exist $C_0\geq0$, $k\geq 0$, $\theta>0$, $\tau>\max\{\frac{\sigma-p+1}{\sigma}\pa{k+1},1\}$  such that
    \begin{equation}\label{e34}
    V(x) \leq C \pa{r(x)}^{C_0}e^{-\theta \pa{\log r(x)}^\tau}
    \end{equation}
    and
    \begin{equation}\label{e35}
     \int_{B_R\setminus B_{R/2}} V^{-\beta}\,d\mu \leq C R^{\alpha}\pa{\log R}^k
    \end{equation}
    for every $R>0$ sufficiently large; then condition \eqref{EQ_hp3} holds.
  \end{itemize}
\end{rem}

We can now state our main theorem.

\begin{theorem}\label{thm_intro3}
Let $p>1$, $\sigma>p-1$, $V\in L^1_{\text{loc}}(M)$ with  $V>0$
a.e. on $M$ and  $a \in \operatorname{Lip}_{\text{loc}}(M)$ with
$a>0$ on $M$. If $u \in W^{1,p}_{\text{loc}}(M) \cap
L^\sigma_{\text{loc}}(M, V d\mu_0)$ is a nonnegative weak solution
of \eqref{EQ_gen}, then $u\equiv0$ on $M$ provided that one of the
conditions (HP1), (HP2) or (HP3) holds (see Definition
\ref{defi_volgrowth}).
\end{theorem}

In the particular case $p=2$, from Theorem \ref{thm_intro3} we can
also derive nonexistence criteria for nonnegative weak solutions
of the semilinear inequality
\begin{equation}\label{EQ_Rig1intro}
    \frac{1}{a(x)}\diver\pa{a(x)\nabla u}+b(x)u+V(x)u^\sigma \leq 0 \quad \text{on }\,  M.
\end{equation}
We refer the reader to Section \ref{RigSec} for a precise
description of the results concerning inequality
\eqref{EQ_Rig1intro}.

The rest of the paper is organized as follows. In Section
\ref{sec2} we state and prove some preliminary technical results,
that we put to use in Section \ref{sec3}, where we give the proof
of Theorem \ref{thm_intro3}. In Section \ref{RigSec} we describe
in more detail nonexistence results for nontrivial nonnegative
weak solutions of \eqref{EQ_Rig1intro}. Finally in Section
\ref{seclast} we collect some counterexamples to Theorem
\ref{thm_intro3} for the case $p=2$, showing that the weighted
volume growth conditions that we assume on geodesic balls are in
many cases sharp.

\vspace{0,3cm}

{\bf Acknowledgements.} The authors wish to thank Prof. Marco
Rigoli for interesting discussions, and in particular, for
suggesting Section \ref{RigSec}.

%
%
%
%
%
%
%
%
%
%
%
%
%

\section{Preliminary results}\label{sec2}

For any relatively compact domain $\Omega\subset M$ and $p>1$,
$W^{1, p}(\Omega)$ is  the completion of the space of Lipschitz
functions $w: \Omega \ra \erre$ with respect to the norm
\[
\|w\|_{W^{1,p}(\Omega)}= \pa{\int_\Omega \abs{\nabla
w}^p\,d\mu_0+\int_\Omega \abs{w}^p\,d\mu_0}^{\frac 1p}.
\]

For any function $u:M\ra\erre$ we say that $u \in
W^{1,p}_{\text{loc}}(M)$ if for every relatively compact domain
$\Omega\subset\subset M$ one has $ u_{|_\Omega} \in
W^{1,p}(\Omega)$.

\begin{defi}
Let $p>1$, $\sigma>p-1$, $V>0$ a.e. on $M$ and $V\in
L^1_{\text{loc}}(M)$. We say that $u$ is a \emph{weak solution } of
equation \eqref{EQ_gen} if $u \in W^{1,p}_{\text{loc}}(M) \cap
L^\sigma_{\text{loc}}(M, V d\mu_0)$ and for every $\varphi \in
W^{1,p}(M) \cap L^\infty(M)$, with $\varphi \geq 0$ a.e. on $M$ and
compact support, one has
 \begin{equation}\label{19}
    -\int_M a(x)\abs{\nabla u}^{p-2}\pair{\nabla u, \nabla\pa{\frac{\varphi}{a(x)}}}\,d\mu_0 + \int_M V(x)u^\sigma \varphi\,d\mu_0 \leq 0 \qquad \text{on }\, M.
  \end{equation}
\end{defi}

\begin{rem}\label{rem1}
  We note that, by \eqref{EQ_propr_a}, $u \in
W^{1,p}_{\text{loc}}(M) \cap L^\sigma_{\text{loc}}(M, V d\mu)$ is a
weak solution of \eqref{EQ_gen} if and only if it  is a weak
solution of
\[
  \diver\pa{a(x)\abs{\nabla u}^{p-2}\nabla u} + a(x) V(x)u^\sigma \leq 0 \quad \text{on }\,
  M,
\]
i.e. if and only if for every $\psi \in W^{1,p}(M) \cap
L^\infty(M)$, with $\psi \geq 0$ a.e. on $M$ and compact support,
one has
\begin{equation}\label{EQ_weakSolGen}
    -\int_M \abs{\nabla u}^{p-2}\pair{\nabla u, \nabla \psi}\,d\mu + \int_M V(x)u^\sigma \psi\,d\mu \leq 0 \qquad \text{on }\, M,
  \end{equation}
where $d\mu$ is the measure on $M$ with density $a$, as defined in
\eqref{meas_a}.

Indeed, given any nonnegative $\psi \in W^{1,p}(M) \cap L^\infty(M)$
with compact support, one can choose $\varphi=a\psi$ as a test
function in \eqref{19} in order to obtain \eqref{EQ_weakSolGen}.
Similarly, , given any nonnegative $\varphi \in W^{1,p}(M) \cap
L^\infty(M)$ with compact support, one can insert
$\psi=\frac{\varphi}{a}$ in \eqref{EQ_weakSolGen} and find
\eqref{19}.
\end{rem}

The following two lemmas will be crucial ingredients in the proof
the Theorem \ref{thm_intro3}.

\begin{lemma}\label{LE_tech1}
Let $s \geq \frac{p\sigma}{\sigma-p+1}$ be fixed. Then there exists
a constant $C>0$ such that for every $t\in\pa{0, \min\set{1, p-1}}$,
every nonnegative weak solution $u$ of equation \eqref{EQ_gen} and
every function $\varphi \in \operatorname{Lip}(M)$ with compact
support and $0 \leq \varphi \leq 1$  one has
  \begin{equation}\label{EQ_2.6Lemma_tech1}
    \frac{t}{p}\int_M{\varphi^su^{-t-1}\abs{\nabla u}^p\chi_\Omega}\,d\mu + \frac{1}{p}\int_M{V u^{\sigma-t}\varphi^s}\,d\mu \leq C t^{-\frac{(p-1)\sigma}{\sigma-p+1}}\int_M{V^{-\frac{p-t-1}{\sigma-p+1}}\abs{\nabla\varphi}^{\frac{p(\sigma-t)}{\sigma-p+1}}}\, d\mu,
  \end{equation}
  where $\Omega = \set{x\in M : u(x)>0}$, $\chi_\Omega$ is the characteristic function of
  $\Omega$ and $d\mu$ is the measure on $M$ with density $a$, as defined in
\eqref{meas_a}.
\end{lemma}

\begin{proof}
Let $\eta>0$ and $u_\eta = u+\eta$, then $u_\eta \in
W^{1,p}_{\text{loc}}(M) \cap L^\sigma_{\text{loc}}(M, V d\mu)$.
Define $\psi = \varphi^su^{-t}_\eta$, then $\psi$ is an admissible
test function for equation \eqref{EQ_weakSolGen}, with
  \begin{equation*}
    \nabla \psi = s\varphi^{s-1}u^{-t}_\eta\nabla\varphi - t\varphi^su^{-t-1}_\eta\nabla u \qquad a.e. \text{ on }\, M.
  \end{equation*}
Indeed, $\operatorname{supp}\psi =\operatorname{supp}\varphi$,
$0\leq \psi \leq \eta^{-t}$ so that $\psi\in L^\infty(M)$ and $\psi
\in W^{1,p}(M)$ with
  \begin{align*}
    \int_M\abs{\nabla \psi}^p\,d\mu &\leq 2^{p-1}\sq{s^p\int_M \varphi^{\pa{s-1}p}u^{-pt}_{\eta}\abs{\nabla \varphi}^p \,d\mu+t^p\int_M{\varphi^{sp}u^{-\pa{t+1}p}_{\eta}\abs{\nabla u}^p} \,d\mu} \\ &\leq 2^{p-1}\sq{s^p\eta^{-pt}\int_M{\abs{\nabla \varphi}^p} \,d\mu+t^p\eta^{-\pa{t+1}p}\int_{\operatorname{supp}\varphi}\abs{\nabla u}^p\,d\mu} < +\infty.
  \end{align*}
Equation \eqref{EQ_weakSolGen} then gives
  \begin{equation}\label{EQ_gianni1}
    t\int_M{\varphi^su^{-t-1}_{\eta}\abs{\nabla u}^p} \,d\mu+\int_M {V u^\sigma u^{-t}_{\eta}\varphi^s}\,d\mu \leq s\int_M {\varphi^{s-1}u^{-t}_{\eta}\abs{\nabla u}^{p-2}\pair{\nabla u, \nabla\varphi}}\,d\mu.
  \end{equation}
Now we estimate the right-hand side of \eqref{EQ_gianni1} using
Young's inequality,  obtaining
  \begin{align*}
    s\int_M {\varphi^{s-1}u^{-t}_{\eta}\abs{\nabla u}^{p-2}\pair{\nabla u, \nabla\varphi}}\,d\mu &\leq s\int_M {\varphi^{s-1}u^{-t}_{\eta}\abs{\nabla u}^{p-1} \abs{\nabla\varphi}}\,d\mu \\&\leq \int_M \pa{t^{\frac{p-1}{p}}\varphi^{s\frac{p-1}{p}}u_\eta^{-\pa{t+1}\frac{p-1}{p}}\abs{\nabla u}^{p-1} }\pa{s t^{-\frac{p-1}{p}}\varphi^{\frac{s}{p}-1}u_\eta^{1-\frac{t+1}{p}}\abs{\nabla \varphi}}\,d\mu \\ &\leq\frac{p-1}{p}\int_M{t\varphi^su^{-t-1}_{\eta}\abs{\nabla u}^p}\,d\mu+\frac{1}{p}\int_M{s^pt^{-(p-1)}\varphi^{s-p}u^{p-(t+1)}_\eta\abs{\nabla\varphi}^p}\,d\mu.
  \end{align*}
From \eqref{EQ_gianni1} we have
    \begin{equation}\label{EQ_gianni2}
    \frac{t}{p}\int_M{\varphi^su^{-t-1}_{\eta}\abs{\nabla u}^p} \,d\mu+\int_M {V u^\sigma u^{-t}_{\eta}\varphi^s}\,d\mu \leq \frac{1}{p}\int_M{s^pt^{-(p-1)}\varphi^{s-p}u^{p-(t+1)}_\eta\abs{\nabla\varphi}^p}\,d\mu.
  \end{equation}
We exploit again Young's inequality on the right-hand side of
\eqref{EQ_gianni2}, with
  \[
  q = \frac{\sigma-t}{p-t-1}, \qquad q'=\frac{q}{q-1} = \frac{\sigma-t}{\sigma-p+1}, \qquad \delta=\frac{p-1}{p}
  \]
obtaining
   \begin{align*}
    \frac{1}{p}\int_M{s^pt^{-(p-1)}\varphi^{s-p}u^{p-(t+1)}_\eta\abs{\nabla\varphi}^p}\,d\mu &= \int_M\pa{\delta^{\frac1q}u_\eta^{p-(t+1)}V^{\frac1q}\varphi^{\frac sq} }\pa{\delta^{-\frac1q}\frac{s^p}{pt^{p-1}} \varphi^{\frac{s}{q'}-p}V^{-\frac1q}\abs{\nabla \varphi}^p}\,d\mu \\ &\leq \frac{\delta}{q}\int_M{u_\eta^{\sigma-t}V \varphi^s }\,d\mu + \frac{1}{q'p^{q'}}\delta^{-\frac{q'}{q}}\pa{\frac{s^p}{t^{p-1}}}^{q'}\int_M {\varphi^{s-pq'} V^{-\frac{q'}{q}}\abs{\nabla \varphi}^{pq'}   }\,d\mu \\&\leq \delta\int_M{u_\eta^{\sigma-t}V \varphi^s }\,d\mu + \delta^{-\frac{p-1}{\sigma-p+1}}\pa{\frac{s^p}{t^{p-1}}}^{\frac{\sigma}{\sigma-p+1}}\int_M { V^{-\frac{q'}{q}}\abs{\nabla \varphi}^{pq'}   }\,d\mu \\&=\frac{p-1}{p}\int_M{u_\eta^{\sigma-t}V \varphi^s }\,d\mu+C t^{-\frac{(p-1)\sigma}{\sigma-p+1}}\int_M { V^{-\frac{p-t-1}{\sigma-p+1}}\abs{\nabla \varphi}^{\frac{p(\sigma-t)}{\sigma-p+1}}   }\,d\mu.
  \end{align*}
Substituting in \eqref{EQ_gianni2} we have
  \begin{align}
    \label{EQ_gianni3}I&=\frac{t}{p}\int_M{\varphi^su^{-t-1}_{\eta}\abs{\nabla u}^p} \,d\mu + \int_M {V u^\sigma u^{-t}_{\eta}\varphi^s}\,d\mu-\frac{p-1}{p}\int_M{V u_\eta^{\sigma-t} \varphi^s }\,d\mu\\
    \nonumber& \leq C t^{-\frac{(p-1)\sigma}{\sigma-p+1}}\int_M { V^{-\frac{p-t-1}{\sigma-p+1}}\abs{\nabla \varphi}^{\frac{p(\sigma-t)}{\sigma-p+1}}   }\,d\mu.
  \end{align}
Since $\nabla u=0$ a.e. on the set $M\setminus \Omega$, see
\cite[Lemma 7.7]{GilTru}, we have
  \begin{align*}
   I= \int_M\sq{\frac{t}{p}\varphi^su^{-t-1}_{\eta}\abs{\nabla u}^p  + V u^\sigma u^{-t}_{\eta}\varphi^s-\frac{p-1}{p}V u_\eta^{\sigma-t} \varphi^s }\chi_\Omega\,d\mu -\int_{M\setminus\Omega}\frac{p-1}{p}V \eta^{\sigma-t} \varphi^s\,d\mu.
  \end{align*}
Now note that $\sq{\frac{t}{p}\varphi^su^{-t-1}_{\eta}\abs{\nabla
u}^p  + V u^\sigma u^{-t}_{\eta}\varphi^s-\frac{p-1}{p}V
u_\eta^{\sigma-t} \varphi^s }\chi_\Omega$ converges a.e. in $M$ to
the function
  \[
  \sq{\frac{t}{p}\varphi^su^{-t-1}\abs{\nabla u}^p  + \frac{1}{p}V u^{\sigma-t} \varphi^s }\chi_\Omega
  \]
as $\eta \ra 0^+$. By an application of Fatou's lemma and using
\eqref{EQ_gianni3} we obtain
\begin{align*}
    \int_M{\frac{t}{p}\varphi^su^{-t-1}\abs{\nabla u}^p\chi_\Omega}\,d\mu &+ \int_M{\frac{1}{p}V u^{\sigma-t} \varphi^s}\,d\mu \\ &=\int_M \sq{\frac{t}{p}\varphi^su^{-t-1}\abs{\nabla u}^p  + \frac{1}{p}V u^{\sigma-t} \varphi^s }\chi_\Omega\,d\mu \\ &\leq \liminf_{\eta\ra0^+} \int_M\sq{\frac{t}{p}\varphi^su^{-t-1}_{\eta}\abs{\nabla u}^p  + V u^\sigma u^{-t}_{\eta}\varphi^s-\frac{p-1}{p}V u_\eta^{\sigma-t} \varphi^s }\chi_\Omega\,d\mu \\&\leq C t^{-\frac{(p-1)\sigma}{\sigma-p+1}}\int_M { V^{-\frac{p-t-1}{\sigma-p+1}}\abs{\nabla \varphi}^{\frac{p(\sigma-t)}{\sigma-p+1}}   }\,d\mu+\liminf_{\eta\ra 0^+}\int_{M\setminus\Omega}\frac{p-1}{p}V \eta^{\sigma-t} \varphi^s\,d\mu \\&=C t^{-\frac{(p-1)\sigma}{\sigma-p+1}}\int_M { V^{-\frac{p-t-1}{\sigma-p+1}}\abs{\nabla \varphi}^{\frac{p(\sigma-t)}{\sigma-p+1}}   }\,d\mu,
\end{align*}
that is inequality \eqref{EQ_2.6Lemma_tech1}.
\end{proof}

\begin{lemma}\label{LE_tech2}
Let $s\geq\frac{2p\sigma}{\sigma-p+1}$ be fixed. Then there exists a
constant $C>0$ such that for every nonnegative weak solution $u$ of
equation \eqref{EQ_gen}, every function $\varphi \in
\operatorname{Lip}(M)$ with compact support and $0 \leq \varphi \leq
1$ and every $t\in(0,\min\{1,p-1,\frac{\sigma-p+1}{2(p-1)}\})$ one
has
\begin{align}
  \label{2.10}\int_M\varphi^su^\sigma V\,d\mu\leq C t^{-\frac{p-1}{p}-\frac{(p-1)^2\sigma}{p(\sigma-p+1)}}&
     \pa{\int_{M\setminus K} V^{-\frac{(t+1)(p-1)}{\sigma-(t+1)(p-1)}}
     \abs{\nabla\varphi}^\frac{p\sigma}{\sigma-(t+1)(p-1)}\,d\mu}^\frac{\sigma-(t+1)(p-1)}{p\sigma}\\
  \nonumber&\pa{\int_M{V^{-\frac{p-t-1}{\sigma-p+1}}\abs{\nabla\varphi}^{\frac{p(\sigma-t)}{\sigma-p+1}}}\,d\mu}^\frac{p-1}{p}
     \pa{\int_{M\setminus K}\varphi^su^\sigma V\,d\mu}^\frac{(t+1)(p-1)}{p\sigma},
\end{align}
with $K=\{x\in M : \varphi(x)=1\}$ and $d\mu$ is the measure on $M$
with density $a$, as defined in \eqref{meas_a}.
\end{lemma}

\begin{proof}[Proof of Lemma \ref{LE_tech2}]

Under our assumptions $\psi=\varphi^s$ is a feasible test function
in equation \eqref{EQ_weakSolGen}. Thus we obtain
\begin{equation}\label{6}
 \int_M\varphi^su^\sigma V\,d\mu\leq \int_M s\varphi^{s-1}\abs{\nabla u}^{p-2}\pair{\nabla u,\nabla\varphi}\,d\mu
   \leq \int_M s\varphi^{s-1}\abs{\nabla
   u}^{p-1}\abs{\nabla\varphi}\,d\mu.
\end{equation}
Now let $\Omega = \set{x\in M : u(x)>0}$ and let $\chi_\Omega$ be
the characteristic function of $\Omega$. Since $\nabla u=0$ a.e. on
the set $M\setminus\Omega$, through an application of H\"{o}lder's
inequality we obtain
\begin{align}
  \label{7}\int_M s\varphi^{s-1}\abs{\nabla u}^{p-1}\abs{\nabla\varphi}\,d\mu&=\int_M s\varphi^{s-1}\abs{\nabla u}^{p-1}\chi_\Omega\abs{\nabla\varphi}\,d\mu\\
  \nonumber&=s\int_M\pa{\varphi^{\frac{p-1}{p}s} \abs{\nabla u}^{p-1}u^{-\frac{p-1}{p}(t+1)}\chi_\Omega}
            \pa{\varphi^{\frac{s}{p}-1}u^{\frac{p-1}{p}(t+1)}\abs{\nabla\varphi}}\,d\mu\\
  \nonumber&\leq s\pa{\int_M\varphi^s \abs{\nabla u}^pu^{-t-1}\chi_\Omega\,d\mu}^\frac{p-1}{p}
            \pa{\int_M\varphi^{s-p}u^{(p-1)(t+1)}\abs{\nabla\varphi}^p\,d\mu}^\frac{1}{p}.
\end{align}
Moreover from equation \eqref{EQ_2.6Lemma_tech1} we deduce
\begin{equation}\label{2.6}
 \int_M{\varphi^s\abs{\nabla u}^pu^{-t-1}\chi_\Omega}\,d\mu \leq C
  t^{-1-\frac{(p-1)\sigma}{\sigma-p+1}}\int_M{V^{-\frac{p-t-1}{\sigma-p+1}}\abs{\nabla\varphi}^{\frac{p(\sigma-t)}{\sigma-p+1}}}\,d\mu,
\end{equation}
with $C>0$ depending on $s$. Thus from \eqref{6}, \eqref{7} and
\eqref{2.6} we obtain
\begin{equation}\label{2.8}
  \int_M\varphi^su^\sigma V\,d\mu\leq C
  \pa{t^{-1-\frac{(p-1)\sigma}{\sigma-p+1}}\int_M{V^{-\frac{p-t-1}{\sigma-p+1}}\abs{\nabla\varphi}^{\frac{p(\sigma-t)}{\sigma-p+1}}}\,d\mu}^\frac{p-1}{p}
  \pa{\int_M \varphi^{s-p} u^{(p-1)(t+1)}\abs{\nabla\varphi}^p\,d\mu}^\frac{1}{p}.
\end{equation}
Now we use again H\"{o}lder's inequality with exponents
\[
q=\frac{\sigma}{(t+1)(p-1)}, \qquad
q'=\frac{q}{q-1}=\frac{\sigma}{\sigma-(t+1)(p-1)}
\]
to obtain
\begin{align*}
 &\int_M \varphi^{s-p}u^{(p-1)(t+1)}\abs{\nabla\varphi}^p\,d\mu\\
 &\qquad=\int_{M\setminus
     K}\pa{\varphi^\frac{s}{q}u^{(p-1)(t+1)}V^\frac{1}{q}}\pa{\varphi^{\frac{s}{q'}-p}V^{-\frac{1}{q}}\abs{\nabla\varphi}^p}\,d\mu\\
 &\qquad\leq \pa{\int_{M\setminus K}\varphi^su^\sigma
     V\,d\mu}^\frac{(t+1)(p-1)}{\sigma}\pa{\int_{M\setminus
     K}\varphi^{s-\frac{p\sigma}{\sigma-(t+1)(p-1)}} V^{-\frac{(t+1)(p-1)}{\sigma-(t+1)(p-1)}}
     \abs{\nabla\varphi}^\frac{p\sigma}{\sigma-(t+1)(p-1)}\,d\mu}^\frac{\sigma-(t+1)(p-1)}{\sigma}.
\end{align*}
Substituting into \eqref{2.8} we get
\begin{align*}
  \int_M\varphi^su^\sigma V\,d\mu\leq C t^{-\frac{p-1}{p}-\frac{(p-1)^2\sigma}{p(\sigma-p+1)}}&
     \pa{\int_M{V^{-\frac{p-t-1}{\sigma-p+1}}\abs{\nabla\varphi}^{\frac{p(\sigma-t)}{\sigma-p+1}}}\,d\mu}^\frac{p-1}{p}
     \pa{\int_{M\setminus K}\varphi^su^\sigma V\,d\mu}^\frac{(t+1)(p-1)}{p\sigma}\\
  \nonumber&\pa{\int_{M\setminus K}\varphi^{s-\frac{p\sigma}{\sigma-(t+1)(p-1)}} V^{-\frac{(t+1)(p-1)}{\sigma-(t+1)(p-1)}}
     \abs{\nabla\varphi}^\frac{p\sigma}{\sigma-(t+1)(p-1)}\,d\mu}^\frac{\sigma-(t+1)(p-1)}{p\sigma}.
\end{align*}
Now inequality \eqref{2.10} immediately follows from the previous
relation, by our assumptions on $s,t$ and since $0\leq\varphi\leq1$.
\end{proof}

From Lemma \ref{LE_tech2} we immediately deduce
\begin{cor}
Under the same assumptions of Lemma \ref{LE_tech2} there exists a
constant $C>0$, independent of $u$, $\varphi$ and $t$, such that
\begin{align}\label{EQ_2.6Lemma_tech2}
   &\pa{\int_M{\varphi^su^{\sigma}V\,d\mu}}^{1-\frac{\pa{t+1}\pa{p-1}}{p\sigma}} \\
   \nonumber &\quad\leq C t^{-\frac{p-1}{p}-\frac{\pa{p-1}^2\sigma}{p\pa{\sigma-p+1}}}
   \pa{\int_M{V^{-\frac{p-t-1}{\sigma-p+1}}\abs{\nabla\varphi}^{\frac{p(\sigma-t)}{\sigma-p+1}}}\, d\mu}^{\frac{p-1}{p}}
   \pa{\int_M V^{-\frac{\pa{t+1}\pa{p-1}}{\sigma-\pa{t+1}\pa{p-1}}}\abs{\nabla\varphi}^{\frac{p\sigma}{\sigma-\pa{t+1}\pa{p-1}}}\,d\mu}^{\frac{\sigma-\pa{t+1}\pa{p-1}}{p\sigma}}.
\end{align}
\end{cor}

\begin{proof}
Inequality \eqref{EQ_2.6Lemma_tech2} easily follows form
\eqref{2.10}, since $s\geq \frac{p\sigma}{\sigma-(t+1)(p-1)}$ and
$0\leq\varphi\leq1$ on $M$.
\end{proof}

\section{Proof of Theorem \ref{thm_intro3}}\label{sec3}

We divide the proof of Theorem \ref{thm_intro3} in three cases,
depending on which of the conditions (HP1), (HP2) or (HP3) is
assumed to hold (see Definition \ref{defi_volgrowth}).

\begin{proof}[Proof of Theorem \ref{thm_intro3}\,.] $(a)$ Assume that condition (HP1) holds (see \eqref{EQ_hp1}). Let
$r(x)$ be the distance of $x\in M$ from a fixed origin $o$, for any
fixed $R>0$ sufficiently large let $t=\frac{1}{\log R}$ and denote
by $B_R$ the metric ball centered at $o$ with radius $R$. Fix any
$C_1\geq\frac{C_0+p+2}{p\sigma}$ with $C_0$ as in condition
\eqref{EQ_hp1}, define for $x\in M$
\begin{equation}\label{8}
 \varphi(x)=\begin{cases}\begin{array}{ll}
   1&\quad\text{for }r(x)<R,\\
   \pa{\frac{r(x)}{R}}^{-C_1t}&\quad\text{for }r(x)\geq R
   \end{array}\end{cases}
\end{equation}
and for $n\in\enne$
\begin{equation}\label{9}
 \eta_n(x)=\begin{cases}\begin{array}{ll}
   1&\quad\text{for }r(x)<nR,\\
   2-\frac{r(x)}{nR}&\quad\text{for }nR\leq r(x)\leq 2nR,\\
   0&\quad\text{for }r(x)\geq2nR.
 \end{array}\end{cases}
\end{equation}
Let
\begin{equation}\label{10}
\varphi_n(x)=\eta_n(x)\varphi(x)\qquad\text{for }x\in M,
\end{equation}
then $\varphi_n\in \operatorname{Lip}(M)$ with
$0\leq\varphi_n\leq1$, we have
$$\nabla\varphi_n=\eta_n\nabla\varphi+\varphi\nabla\eta_n\qquad \text{a.e. in }M$$ and for
every $a\geq1$ $$|\nabla\varphi_n|^a\leq
2^{a-1}\pa{|\nabla\varphi|^a+\varphi^a|\nabla\eta_n|^a}\qquad
\text{a.e. in }M.$$ Now we use $\varphi_n$ in formula
\eqref{EQ_2.6Lemma_tech1} of Lemma \ref{LE_tech1} with any fixed
$s\geq\frac{p\sigma}{\sigma-p+1}$ and deduce that, for some positive
constant $C$ and for every $n\in\enne$ and every small enough $t>0$,
we have
\begin{align}
 \label{20}\int_M{V u^{\sigma-t}\varphi_n^s}\,d\mu& \leq
   Ct^{-\frac{(p-1)\sigma}{\sigma-p+1}}\int_M{V^{-\frac{p-t-1}{\sigma-p+1}}\abs{\nabla\varphi_n}^{\frac{p(\sigma-t)}{\sigma-p+1}}}\,d\mu\\
 \nonumber &= C  t^{-\frac{(p-1)\sigma}{\sigma-p+1}} \int_M
   V^{-\beta+\frac{t}{\sigma-p+1}}|\nabla\varphi_n^{\frac{p(\sigma-t)}{\sigma-p+1}}d\mu\\
 \nonumber&\leq Ct^{-\frac{(p-1)\sigma}{\sigma-p+1}}2^{\frac{p(\sigma-t)}{\sigma-p+1}-1}
   \sq{\int_M{V^{-\beta+\frac{t}{\sigma-p+1}}\abs{\nabla\varphi}^{\frac{p(\sigma-t)}{\sigma-p+1}}}\,d\mu+
   \int_{B_{2nR}\setminus
   B_{nR}}{V^{-\beta+\frac{t}{\sigma-p+1}}\varphi^{\frac{p(\sigma-t)}{\sigma-p+1}}}\abs{\nabla\eta_n}^{\frac{p(\sigma-t)}{\sigma-p+1}}\,d\mu}\\
 \nonumber&\leq Ct^{-\frac{(p-1)\sigma}{\sigma-p+1}}\sq{I_1+I_2},
\end{align}
where
\begin{align*}
 I_1&:= \int_{M\setminus B_R}V^{-\beta+\frac{t}{\sigma-p+1}}|\nabla\varphi|^{\frac{p(\sigma-t)}{\sigma-p+1}}d\mu,\\
 I_2&:=  \int_{B_{2nR}\setminus B_{nR}}\varphi^{\frac{p(\sigma-t)}{\sigma-p+1}}|\nabla\eta_n|^{\frac{p(\sigma-t)}{\sigma-p+1}}V^{-\beta+\frac{t}{\sigma-p+1}}d\mu.
\end{align*}
By \eqref{8}, \eqref{9} and assumption (HP1) with
$\eps=\frac{t}{\sigma-p+1}$, see equation \eqref{EQ_hp1}, for every
$n\in\enne$ and every small enough $t>0$ we have
\begin{align}
 \label{217a}I_2 & \leq \left(\sup_{B_{2nR}\setminus B_{nR}} \varphi\right)^{\frac{p(\sigma-t)}{\sigma-p+1}}\left(\frac 1{nR}\right)^{\frac{p(\sigma-t)}{\sigma-p+1}}
   \int_{B_{2nR}\setminus B_{nR}} V^{-\beta + \frac{t}{\sigma-p+1}}\,d\mu \\
 \nonumber& \leq C \left(\frac{nR}{R}\right)^{-\frac{p(\sigma-t)}{\sigma-p+1}C_1 t}\left(\frac 1{nR} \right)^{\frac{p(\sigma-t)}{\sigma-p+1}}
   (2nR)^{\alpha+ \frac{C_0t}{\sigma-p+1}}[\log(2nR)]^k \\
 \nonumber& \leq C n^{\alpha+\frac{C_0 t}{\sigma-p+1}-\frac{p(\sigma-t)}{\sigma-p+1}(C_1t+1)} R^{\alpha+\frac{C_0 t}{\sigma-p+1}
   -\frac{p(\sigma-t)}{\sigma-p+1}}[\log(2n R)]^k\,.
\end{align}
By our choice of $C_1$, for every small enough $t>0$
\begin{equation}\label{217b}
\alpha+\frac{C_0 t}{\sigma-p+1}-
\frac{p(\sigma-t)}{\sigma-p+1}(C_1t+1)= \frac{t(C_0 -p\sigma C_1 +
pC_1 t+p)}{\sigma-p+1}\leq -\frac{t}{\sigma-p+1}<0.
\end{equation}
Moreover, since $t=\frac{1}{\log R}$, we have
\[
 R^{\alpha+\frac{C_0 t}{\sigma-p+1}
   -\frac{p(\sigma-t)}{\sigma-p+1}}=R^{\frac{C_0+p}{\sigma-p+1}t}=e^{\frac{C_0+p}{\sigma-p+1}t\log
 R}=e^{\frac{C_0+p}{\sigma-p+1}}
\]
In view of \eqref{217a} and \eqref{217b} for $R>1$ large enough, and
thus $t=\frac 1{\log R}$ small enough, we obtain
\begin{equation}\label{218}
I_2 \leq C n^{-\frac{t}{\sigma-p+1}}[\log(2n R)]^k\,.
\end{equation}

In order to estimate $I_1$ we recall that if
$f:[0,\infty)\rightarrow[0,\infty)$ is a nonnegative decreasing
function and \eqref{EQ_hp1} holds, then for any small enough
$\eps>0$ and any sufficiently large $R>1$ we have
\begin{equation}\label{219}
\int_{M\setminus B_R} f(r(x)) \pa{V(x)}^{-\beta+\eps}\,d\mu\leq
C\int_{\frac{R}{2}}^{+\infty}f(r)r^{\alpha+C_0\eps-1}(\log r)^k\,dr
\end{equation}
for some positive constant $C$, see \cite[formula (2.19)]{GrigS}.
Moreover, there holds
\begin{equation}\label{219a}
|\nabla \varphi| \leq C_1 t R^{C_1 t} r^{-C_1 t -1}\,.
\end{equation}
Thus, using \eqref{219}-\eqref{219a},
\begin{align*}
  I_1 &\leq \int_{M\setminus B_R} V^{-\beta+\frac{t}{\sigma-p+1}}(R^{C_1 t}C_1 t r^{-C_1t-1})^{\frac{p(\sigma-t)}{\sigma-p+1}}\, d\mu\\
  &\leq C \int_{\frac{R}{2}}^\infty R^{\frac{p(\sigma-t)}{\sigma-p+1}C_1 t} (1+C_1)^{\frac{p\sigma}{\sigma-p+1}} t^{\frac{p(\sigma-t)}{\sigma-p+1}}
     r^{-\frac{p(\sigma-t)}{\sigma-p+1}(C_1t+1)+\alpha+C_0\frac{t}{\sigma-p+1}-1}(\log r)^k\, dr.
\end{align*}
Now note that $$R^{\frac{p(\sigma-t)}{\sigma-p+1}C_1
t}=e^{\frac{p(\sigma-t)}{\sigma-p+1}C_1}<e^\frac{p\sigma
C_1}{\sigma-p+1}$$ and that by our choice of $C_1$ we have
\[a:= -\frac{p(\sigma-t)}{\sigma-p+1}(C_1t+1)+\alpha+C_0\frac{t}{\sigma-p+1}=\frac t{\sigma-p+1}(pC_1 t-p\sigma C_1 +p+C_0)\leq -\frac{t}{\sigma-p+1}<0\,.\]
Then, by the above inequalities and performing the change of
variables $\xi:=|a|\log r$, we get
\begin{align}
  \label{11}I_1& \leq C t^{\frac{p(\sigma-t)}{\sigma-p+1}}\int_{1}^\infty
     r^{-\frac{p(\sigma-t)}{\sigma-p+1}(C_1t+1)+\alpha+C_0\frac{t}{\sigma-p+1}}(\log r)^k\,\frac{dr}{r}\\
  \nonumber&\leq C |a|^{-(k+1)} t^{\frac{p(\sigma-t)}{\sigma-p+1}} \int_0^\infty e^{-\xi} \xi^k\, d\xi\\
  \nonumber&\leq C\pa{\frac{t}{\sigma-p+1}}^{-k-1}t^{\frac{p(\sigma-t)}{\sigma-p+1}} \int_0^\infty e^{-\xi} \xi^k\,
    d\xi\\
  \nonumber&\leq C t^{\frac{p(\sigma-t)}{\sigma-p+1}-k-1}\,.
\end{align}
By \eqref{20}, \eqref{218} and \eqref{11}
\begin{align}
  \label{219g}\int_{B_R} V u^{\sigma-t}\, d\mu\,& \leq \int_M V u^{\sigma-t}\varphi_n^s\, d\mu\\
  \nonumber& \leq C t^{-\frac{(p-1)\sigma}{\sigma-p+1}}[ n^{-\frac{t}{\sigma-p+1}}(\log(2n R))^k
    + t^{\frac{p(\sigma-t)}{\sigma-p+1}-k-1}].
\end{align}
Since $R>1$ is large and fixed, and thus $t=\frac 1{\log R} < 1$
is also fixed, taking the $\liminf$ as $n\to \infty$ in
\eqref{219g} we obtain
\begin{equation}\label{219h}
\int_{B_R} V u^{\sigma-t} d\mu \leq C t^{\frac{p(\sigma-t)}{\sigma-p+1}-k-1-\frac{(p-1)\sigma}{\sigma-p+1}}.
\end{equation}
Observe that, for each small enough $t>0$,
$$\frac{p(\sigma-t)}{\sigma-p+1}-k-1-\frac{(p-1)\sigma}{\sigma-p+1} = \frac{p-1}{\sigma-p+1}-k-\frac{pt}{\sigma-p+1}=\beta-k-\frac{pt}{\sigma-p+1}\geq \delta_*>0\,.$$
Then, for any fixed sufficiently small $t>0$, we have
\[\int_M V u^{\sigma-t}\chi_{B_{e^{1/t}}}\,d\mu = \int_{B_R} V u^{\sigma-t}\, d\mu \leq C t^{\delta_*}\,. \]
By Fatou's Lemma, taking the $\liminf$ as $t\to 0^+$ in the previous
inequality we obtain
\[ \int_M V u^\sigma\, d\mu \leq 0,\]
which implies $u\equiv 0$ in $M$.

\medskip

\noindent $(b)$ Assume that condition (HP2) holds (see
\eqref{EQ_hp2}). Let the functions $\varphi$, $\eta_n$ and
$\varphi_n$ be defined on $M$ as in formulas \eqref{8}, \eqref{9}
and \eqref{10}, with $R>1$ large enough, $t=\frac{1}{\log R}$,
$C_1\geq\max\set{\frac{C_0+p+2}{p\sigma},\frac{C_0}{\sigma-p+1}}$
and $C_0$ as in condition \eqref{EQ_hp2}. We now apply formula
\eqref{EQ_2.6Lemma_tech2}, using the family of functions
$\varphi_n\in \operatorname{Lip}_0(M)$ and any fixed
$s\geq\frac{2p\sigma}{\sigma-p+1}$, and thus we have
\begin{align*}
   &\pa{\int_M{\varphi_n^su^{\sigma}V\,d\mu}}^{1-\frac{\pa{t+1}\pa{p-1}}{p\sigma}} \\
   \nonumber &\quad\leq C t^{-\frac{p-1}{p}-\frac{\pa{p-1}^2\sigma}{p\pa{\sigma-p+1}}}
   \pa{\int_M{V^{-\frac{p-t-1}{\sigma-p+1}}\abs{\nabla\varphi_n}^{\frac{p(\sigma-t)}{\sigma-p+1}}}\, d\mu}^{\frac{p-1}{p}}
   \pa{\int_M V^{-\frac{\pa{t+1}\pa{p-1}}{\sigma-\pa{t+1}\pa{p-1}}}\abs{\nabla\varphi_n}^{\frac{p\sigma}{\sigma-\pa{t+1}\pa{p-1}}}\,d\mu}^{\frac{\sigma-\pa{t+1}\pa{p-1}}{p\sigma}}.
\end{align*}
We now need need to estimate
  \begin{equation}\label{EQ_2.11in2_a}
    \int_M V^{-\frac{p-t-1}{\sigma-p+1}}\abs{\nabla
    \varphi}^{\frac{p(\sigma-t)}{\sigma-p+1}}\,d\mu\qquad\text{ and
    }\qquad
    \int_M V^{-\frac{(t+1)(p-1)}{\sigma-\pa{t+1}\pa{p-1}}}\abs{\nabla \varphi}^{\frac{p\sigma}{\sigma-\pa{t+1}\pa{p-1}}}\,d\mu.
  \end{equation}
Arguing as in the previous proof of the theorem under the validity
of condition (HP1), with the only difference that the condition
$k<\beta$ there is replaced here by $k=\beta$, using \eqref{EQ_hp2}
we can deduce that
  \begin{equation}\label{15}
  \int_M V^{-\frac{p-t-1}{\sigma-p+1}}\abs{\nabla \varphi_n}^{\frac{p(\sigma-t)}{\sigma-p+1}}\,d\mu\leq C\sq{n^{-\frac{t}{\sigma-p+1}}(\log\pa{2nR})^\beta
  +t^{\frac{p(\sigma-t)}{\sigma-p+1}-\beta-1}}.
  \end{equation}
In order to estimate the second integral in \eqref{EQ_2.11in2_a} we
start by defining $\Lambda = \frac{(p-1)\sigma
t}{(\sigma-p+1)\sq{\sigma-\pa{t+1}\pa{p-1}}}$, and we note that
\begin{equation}\label{12}
   \frac{(p-1)\sigma}{\pa{\sigma-p+1}^2}t<\Lambda<\frac{2(p-1)\sigma}{\pa{\sigma-p+1}^2}t<\eps^*
\end{equation}
for every small enough $t>0$, and that
  \[\frac{(t+1)(p-1)}{\sigma-\pa{t+1}\pa{p-1}}=\beta+\Lambda
  \qquad\text{ and }\qquad\frac{p\sigma}{\sigma-\pa{t+1}\pa{p-1}}=\alpha +\Lambda
  p,
  \]
with $\alpha,\beta$ as in Definition \ref{defi_volgrowth}. By our
definition of the functions $\varphi_n$, for every $n\in\enne$ and
every small enough $t>0$ we have
\begin{align}
  \label{14}\int_M V^{-\beta-\Lambda}\abs{\nabla \varphi_n}^{\alpha+\Lambda p}\,d\mu& \leq
    C\sq{\int_M V^{-\beta-\Lambda}{\eta_n}^{\alpha+\Lambda p}\abs{\nabla \varphi}^{\alpha+\Lambda p}\,d\mu
      +\int_M V^{-\beta-\Lambda} \varphi^{\alpha+\Lambda p}\abs{\nabla \eta_n}^{\alpha+\Lambda p}\,d\mu}\\
  \nonumber  &\leq C\sq{\int_{M\setminus B_R} V^{-\beta-\Lambda} \abs{\nabla \varphi}^{\alpha+\Lambda p}\,d\mu
      +\int_{B_{2nR}\setminus B_{nR}} V^{-\beta-\Lambda} \varphi^{\alpha+\Lambda p}\abs{\nabla\eta_n}^{\alpha+\Lambda p}\,d\mu}\\
   \nonumber &:=C\pa{I_1+I_2}.
\end{align}
Now we use condition \eqref{EQ_hp2} with $\eps=\Lambda$, and we
obtain
   \begin{align*}
    I_2= \int_{B_{2nR}\setminus B_{nR}} V^{-\beta-\Lambda} \varphi^{\alpha+\Lambda p}\abs{\nabla \eta_n}^{\alpha+\Lambda p}\,d\mu
       &\leq \pa{\sup_{B_{2nR}\setminus B_{nR}}\varphi}^{\alpha+\Lambda p}\pa{\frac{1}{nR}}^{\alpha+\Lambda p}\pa{ \int_{B_{2nR}\setminus B_{nR}}
           V^{-\beta-\Lambda}\,d\mu} \\
    &\leq Cn^{-\pa{\alpha+\Lambda p}C_1t}\pa{\frac{1}{nR}}^{\alpha+\Lambda p}\pa{2nR}^{\alpha+C_0\Lambda}\pa{\log\pa{2nR}}^\beta \\
    &\leq C n^{-\pa{\alpha+\Lambda p}C_1t-p\Lambda +C_0\Lambda }R^{-p\Lambda + C_0\Lambda}\pa{\log\pa{2nR}}^\beta.
   \end{align*}
By our definition of $C_1,\Lambda$ and by relation \eqref{12} we
easily find
   \begin{align}
     \label{13} -C_1\pa{\alpha+\Lambda p}t-\Lambda p+\Lambda C_0 &< -\frac{p\sigma tC_1}{\sigma-\pa{t+1}\pa{p-1}}+\frac{\sigma t(p-1)C_0}{\sq{\sigma-\pa{t+1}\pa{p-1}}\pa{\sigma-p+1}} \\
     \nonumber&\leq-\frac{\sigma tC_0}{\sq{\sigma-\pa{t+1}\pa{p-1}}\pa{\sigma-p+1}} < -\frac{\sigma tC_0}{\pa{\sigma-p+1}^2} <0,
   \end{align}
for any small enough $t>0$. Moreover by \eqref{12}, since
$t=\frac{1}{\log R}$, we have
\[
R^{-p\Lambda+C_0\Lambda}\leq R^{C_0\Lambda}\leq
R^{\frac{2(p-1)\sigma C_0t}{(\sigma-p-1)^2}}=e^\frac{2(p-1)\sigma
C_0}{(\sigma-p-1)^2}.
\]
Thus, for any sufficiently large $R>0$,
\begin{equation}\label{EQ_EstI2}
    I_2 \leq C n^{ -\frac{\sigma tC_0}{\pa{\sigma-p+1}^2}}\pa{\log\pa{2nR}}^\beta.
\end{equation}
In order to estimate $I_1$ we note that if
$f:[0,\infty)\rightarrow[0,\infty)$ is a nonnegative decreasing
function and \eqref{EQ_hp2} holds, then for any small enough
$\eps>0$ and any sufficiently large $R>1$ we have
\begin{equation}\label{219}
\int_{M\setminus B_R} f(r(x)) \pa{V(x)}^{-\beta-\eps}\,d\mu\leq
C\int_{\frac{R}{2}}^{+\infty}f(r)r^{\alpha+C_0\eps-1}(\log
r)^\beta\,dr
\end{equation}
for some positive constant $C$, see \eqref{219} and \cite[formula
(2.19)]{GrigS}. Thus, noting that $\abs{\nabla \varphi}\leq C_1
tR^{C_1t}r^{-C_1t-1}$ a.e. on $M$ and using \eqref{12}, for every
small enough $t>0$ we have
  \begin{align*}
    I_1 &\leq \int_{M\setminus B_R} V^{-\beta-\Lambda}\pa{C_1 tR^{C_1t}r^{-C_1t-1}}^{\alpha+\Lambda p}\,d\mu \\
    &\leq C\int_{\frac R2}^\infty R^{(\alpha+\Lambda p)C_1t}\pa{1+C_1}^{\alpha+\Lambda p}t^{\alpha+\Lambda p}
       r^{-(\alpha+\Lambda p)(C_1t+1)+\alpha+C_0\Lambda-1}(\log r)^\beta\,dr \\
  \end{align*}
  Now, since $t=\frac{1}{\log R}$, by relation \eqref{12}we have
  \[
      R^{(\alpha+\Lambda p)C_1t}= e^{(\alpha+\Lambda p)C_1}\leq
      e^{(\alpha+\eps^*p)C_1};
  \]
  moreover, as we noted already in \eqref{13}, for $t>0$ small enough
  \[
  b=-(\alpha+\Lambda
  p)(C_1t+1)+\alpha+C_0\Lambda<-\frac{C_0t\sigma}{\pa{\sigma-p+1}^2}<0.
  \]
  With the change of variables $\xi=\abs{b}\log r$, using the previous relations we find
  \begin{align}
    \label{EQ_EstI1}  I_1 &\leq C t^{\alpha+\Lambda p}\int_1^\infty r^b\pa{\log r}^\beta\,\frac{dr}{r}= C t^{\alpha+\Lambda p}\abs{b}^{-\beta-1}\pa{\int_0^\infty
       e^{-\xi}\xi^\beta\,d\xi}\\
    \nonumber&\leq C\pa{\frac{\pa{\sigma-p+1}^2}{C_0\sigma}}^{\beta+1}t^{\alpha+\Lambda p-\beta-1}=Ct^{\alpha+\Lambda  p-\beta-1}.
  \end{align}
  From equations \eqref{14}, \eqref{EQ_EstI2} and \eqref{EQ_EstI1} it follows that
  \begin{align}\label{EQ_EstI1plusI2}
      \int_M V^{-\beta-\Lambda}\abs{\nabla \varphi_n}^{\alpha+\Lambda p}\,d\mu&\leq C\sq{t^{\alpha+\Lambda  p-\beta-1}+n^{ -\frac{\sigma tC_0}{\pa{\sigma-p+1}^2}}\pa{\log\pa{2nR}}^\beta}.
  \end{align}
  From \eqref{EQ_2.6Lemma_tech2}, using \eqref{15} and \eqref{EQ_EstI1plusI2} then we have
  \begin{align*}
    \pa{\int_{B_R}u^\sigma V\,d\mu}^{1-\frac{(t+1)(p-1)}{p\sigma}} &\leq \pa{\int_{M}\varphi_n^su^\sigma V\,d\mu}^{1-\frac{(t+1)(p-1)}{p\sigma}}\\
    &\leq C t^{-\frac{p-1}{p}-\frac{\pa{p-1}^2\sigma}{p\pa{\sigma-p+1}}}\pa{\int_{M}V^{-\frac{p-t-1}{\sigma-p+1}}
       \abs{\nabla\varphi_n}^{\frac{p\pa{\sigma-t}}{\sigma-p+1}}\,d\mu}^{\frac{p-1}{p}} \\
    &\,\,\,\,\,\, \times\pa{\int_MV^{-\beta-\Lambda}\abs{\nabla\varphi_n}^{\alpha+\Lambda p}\,d\mu}^\frac{1}{\alpha+\Lambda p} \\
    &\leq C t^{-\frac{p-1}{p}-\frac{\pa{p-1}^2\sigma}{p\pa{\sigma-p+1}}}\sq{n^{-\frac{t}{\sigma-p+1}}(\log\pa{2nR})^\beta
       +t^{\frac{p(\sigma-t)}{\sigma-p+1}-\beta-1}}^{\frac{p-1}{p}} \\
    &\,\,\,\,\,\,\times\sq{t^{\alpha+\Lambda  p-\beta-1}+n^{ -\frac{\sigma tC_0}{\pa{\sigma-p+1}^2}}\pa{\log\pa{2nR}}^\beta}^{\frac{1}{\alpha+\Lambda p}}
  \end{align*}
  By taking the $\liminf$ as $n\ra+\infty$ we get
  \begin{align*}
  \pa{\int_{B_R}u^\sigma V\,d\mu}^{1-\frac{(t+1)(p-1)}{p\sigma}} &\leq C
    t^{-\frac{p-1}{p}-\frac{\pa{p-1}^2\sigma}{p\pa{\sigma-p+1}}+\frac{(p-1)(\sigma-t)}{\sigma-p+1}-\frac{(\beta+1)(p-1)}{p}+1-\frac{\beta+1}{\alpha+\Lambda p}}
  \end{align*}
  for every sufficiently small $t>0$, with $t=\frac{1}{\log R}$. But
  \[
  -\frac{p-1}{p}-\frac{\pa{p-1}^2\sigma}{p\pa{\sigma-p+1}}+\frac{(p-1)(\sigma-t)}{\sigma-p+1}-\frac{(\beta+1)(p-1)}{p}+1-\frac{\beta+1}{\alpha+\Lambda p}
  =-\frac{(p-1)^2}{p\pa{\sigma-p+1}}t,
  \]
 hence for every small enough $t>0$ we have
  \begin{equation*}
    \pa{\int_{B_{e^{1/t}}}u^\sigma V\,d\mu}^{1-\frac{(t+1)(p-1)}{p\sigma}} \leq C
    t^{-\frac{(p-1)^2}{p\pa{\sigma-p+1}}t}\leq C,
  \end{equation*}
  that is
  \begin{equation*}
    \int_{B_{e^{1/t}}}u^\sigma V\,d\mu \leq C
  \end{equation*}
  uniformly in $t$, for $t>0$ sufficiently small. By taking the limit for $t\ra0^+$ we deduce
  \begin{equation}
    \int_{M}u^\sigma V\,d\mu < +\infty,
  \end{equation}
  and thus $u\in L^\sigma(M,Vd\mu)$. Now we exploit inequality \eqref{2.10} with the cutoff function
  $\varphi_n$, and using again \eqref{15} and \eqref{EQ_EstI1plusI2} we obtain
  \begin{align*}
    \int_M\varphi_n^s u^\sigma V\,d\mu &\leq C t^{-\frac{p-1}{p}-\frac{\pa{p-1}^2\sigma}{p\pa{\sigma-p+1}}}
    \sq{n^{-\frac{t}{\sigma-p+1}}(\log\pa{2nR})^\beta+t^{\frac{p(\sigma-t)}{\sigma-p+1}-\beta-1}}^{\frac{p-1}{p}}
    \pa{\int_{M\setminus B_R}\varphi_n^s u^\sigma V\,d\mu}^{\frac{(t+1)(p-1)}{p\sigma}} \\
    &\,\,\,\,\,\,\times\sq{t^{\alpha+\Lambda  p-\beta-1}+n^{ -\frac{\sigma tC_0}{\pa{\sigma-p+1}^2}}\pa{\log\pa{2nR}}^\beta}^{\frac{1}{\alpha+\Lambda p}}.
  \end{align*}
  Since $\varphi_n\equiv 1$ on $B_R$ and $0<\varphi_n\leq 1$ on $M$, for all $n\in\mathds{N}$
  \[
   \int_{B_R} u^\sigma V\,d\mu\leq \int_M\varphi_n^s u^\sigma V\,d\mu, \qquad \int_{M\setminus B_R}\varphi_n^s u^\sigma V\,d\mu \leq \int_{M\setminus B_R} u^\sigma V\,d\mu.
  \]
  Using previous inequalities and taking the $\liminf$ as $n\ra+\infty$ we
  get
  \begin{align*}
    \int_{B_R} u^\sigma V\,d\mu &\leq C t^{-\frac{p-1}{p}-\frac{\pa{p-1}^2\sigma}{p\pa{\sigma-p+1}}+\frac{(p-1)(\sigma-t)}{\sigma-p+1}-\frac{(\beta+1)(p-1)}{p}
    +1-\frac{\beta+1}{\alpha+\Lambda p}}\pa{\int_{M\setminus B_R} u^\sigma V\,d\mu}^{\frac{(t+1)(p-1)}{p\sigma}} \\
    &= C t^{-\frac{(p-1)^2}{p\pa{\sigma-p+1}}t}\pa{\int_{M\setminus B_R} u^\sigma V\,d\mu}^{\frac{(t+1)(p-1)}{p\sigma}} \leq C \pa{\int_{M\setminus B_R} u^\sigma V\,d\mu}^{\frac{(t+1)(p-1)}{p\sigma}}
  \end{align*}
  uniformly for $t>0$ sufficiently small, with $t=\frac{1}{\log R}$. Since $u\in L^\sigma(M,Vd\mu)$,
  \[
  \int_{M\setminus B_R} u^\sigma V\,d\mu \ra 0 \quad \text{ as }\, R\ra +\infty.
  \]
  Moreover $\frac{(t+1)(p-1)}{p\sigma}\ra\frac{p-1}{p\sigma}>0$ as $R\ra +\infty$. It follows that
  \[
   \int_{M} u^\sigma V\,d\mu = \lim_{R\ra +\infty}\int_{B_R}u^\sigma V\,d\mu=0,
  \]
  which implies $u\equiv0$ on $M$.

\medskip

\noindent $(c)$ Assume that condition (HP3) holds (see
\eqref{EQ_hp3}). Consider the functions $\varphi$, $\eta_n$ and
$\varphi_n$ defined in \eqref{8}, \eqref{9} and \eqref{10}, with
$R>0$ large enough, $t=\frac{1}{\log R}$,
$C_1\geq\frac{C_0+p+2}{p\sigma}$ and $C_0$ as in condition
\eqref{EQ_hp3}. Arguing as in the previous proof of the theorem
under the assumption of the validity of (HP1), by formula
\eqref{EQ_2.6Lemma_tech1} with any fixed
$s\geq\frac{p\sigma}{\sigma-p+1}$, we see that
\begin{align}
 \label{2}\int_M{V u^{\sigma-t}\varphi_n^s}\,d\mu& \leq
   Ct^{-\frac{(p-1)\sigma}{\sigma-p+1}}\int_M{V^{-\frac{p-t-1}{\sigma-p+1}}\abs{\nabla\varphi_n}^{\frac{p(\sigma-t)}{\sigma-p+1}}}\,d\mu\\
 \nonumber&\leq Ct^{-\frac{(p-1)\sigma}{\sigma-p+1}}
   \sq{\int_M{V^{-\frac{p-t-1}{\sigma-p+1}}\abs{\nabla\varphi}^{\frac{p(\sigma-t)}{\sigma-p+1}}}\,d\mu+
   \int_{B_{2nR}\setminus
   B_{nR}}{V^{-\frac{p-t-1}{\sigma-p+1}}\varphi^{\frac{p(\sigma-t)}{\sigma-p+1}}}\abs{\nabla\eta_n}^{\frac{p(\sigma-t)}{\sigma-p+1}}\,d\mu}\\
 \nonumber&:=Ct^{-\frac{(p-1)\sigma}{\sigma-p+1}}\sq{I_1+I_2},
\end{align}
for some positive constant $C$ and for every $n\in\enne$ and every
small enough $t>0$. Now, recalling the definitions of $\varphi$ and
$\eta_n$, by condition \eqref{EQ_hp3} with
$\eps=\frac{t}{\sigma-p+1}$, for every small enough $t>0$ we have
\begin{align*}
  I_2&\leq \pa{\sup_{B_{2nR}\setminus B_{nR}}\varphi}^{\frac{p(\sigma-t)}{\sigma-p+1}}\pa{\frac{1}{nR}}^{\frac{p(\sigma-t)}{\sigma-p+1}}
     \int_{B_{2nR}\setminus B_{nR}}V^{-\beta+\frac{t}{\sigma-p+1}}\,d\mu\\
  &\leq Cn^{-\frac{p(\sigma-t)}{\sigma-p+1}C_1t}\pa{\frac{1}{nR}}^{\frac{p(\sigma-t)}{\sigma-p+1}}
     \pa{2nR}^{\alpha+\frac{C_0t}{\sigma-p+1}}\pa{\log (2nR)}^k e^{-\frac{\theta t}{\sigma-p+1} \pa{\log (2nR)}^\tau}\\
  &=Cn^{\frac{t}{\sigma-p+1}\pa{C_0-p\sigma C_1+pC_1t+p}}R^{\frac{C_0+p}{\sigma-p+1}t}\pa{\log (2nR)}^k e^{-\frac{\theta t}{\sigma-p+1}\pa{\log
     (2nR)}^\tau}.
\end{align*}
Note that, since $t=\frac{1}{\log R}$, we have
\[
 R^{\frac{C_0+p}{\sigma-p+1}t}=e^{\frac{C_0+p}{\sigma-p+1}t\log
 R}=e^{\frac{C_0+p}{\sigma-p+1}}
\]
and that by our choice of $C_1$, if $t>0$ is sufficiently small, we
have $\pa{C_0-p\sigma C_1+pC_1t+p}<-1$. Thus we conclude that
\begin{equation}\label{1}
I_2\leq C n^{-\frac{t}{\sigma-p+1}}\pa{\log (2nR)}^k.
\end{equation}

In order to estimate $I_1$ we note that if
$f:[0,\infty)\rightarrow[0,\infty)$ is a nonnegative decreasing
function and \eqref{EQ_hp3} holds, then for any small enough
$\eps>0$ and any sufficiently large $R>0$ we have
\begin{equation}\label{2.19}
\int_{M\setminus B_R} f(r(x)) \pa{V(x)}^{-\beta+\eps}\,d\mu\leq
C\int_{\frac{R}{2}}^{+\infty}f(r)r^{\alpha+C_0\eps-1}(\log
r)^ke^{-\eps\theta(\log r)^\tau}\,dr
\end{equation}
for some positive fixed constant $C$. Indeed, by the monotonicity of
the involved functions, using condition \eqref{EQ_hp3} we obtain in
a similar way as \cite[formula (2.19)]{GrigS}
\begin{align*}
 \int_{M\setminus B_R} f(r(x)) \pa{V(x)}^{-\beta+\eps}\,d\mu
   &=\sum_{i=0}^{+\infty}\int_{B_{2^{i+1}R}\setminus B_{2^iR}}f(r(x))\pa{V(x)}^{-\beta+\eps}\,d\mu\\
   &\leq\sum_{i=0}^{+\infty}f(2^iR)\int_{B_{2^{i+1}R}\setminus B_{2^iR}}V^{-\beta+\eps}\,d\mu\\
   &\leq C\sum_{i=0}^{+\infty}f(2^iR)e^{-\eps\theta(\log(2^{i+1}R))^\tau}(2^{i+1}R)^{\alpha+C_0\eps}(\log(2^{i+1}R))^k\\
   &\leq C\sum_{i=0}^{+\infty}f(2^iR)e^{-\eps\theta(\log(2^{i+1}R))^\tau}(2^{i-1}R)^{\alpha+C_0\eps}(\log(2^{i-1}R))^k\\
   &\leq C\sum_{i=0}^{+\infty}\int_{2^{i-1}R}^{2^iR}f(r)e^{-\eps\theta(\log r)^\tau}r^{\alpha+C_0\eps-1}(\log r)^k\,dr\\
   &= C\int_{\frac{R}{2}}^{+\infty}f(r)e^{-\eps\theta(\log r)^\tau}r^{\alpha+C_0\eps-1}(\log r)^k\,dr.
\end{align*}

Now, since for a.e. $x\in M$ we have $$\abs{\nabla\varphi(x)}\leq
C_1tR^{C_1t}(r(x))^{-C_1t-1},$$ using \eqref{2.19} with
$\eps=\frac{t}{\sigma-p+1}$, we obtain that for every small enough
$t>0$ with $t=\frac{1}{\log R}$
\begin{align*}
  I_1&\leq \int_{M\setminus B_R}
     V^{-\beta+\frac{t}{\sigma-p+1}}\pa{C_1tR^{C_1t}(r(x))^{-C_1t-1}}^\frac{p(\sigma-t)}{\sigma-p+1}\,d\mu\\
  &\leq C\int_{\frac{R}{2}}^{+\infty}R^{\frac{p(\sigma-t)C_1t}{\sigma-p+1}}C_1^\frac{p(\sigma-t)}{\sigma-p+1}
     t^\frac{p(\sigma-t)}{\sigma-p+1}r^{-\frac{p(\sigma-t)(C_1t+1)}{\sigma-p+1}+\alpha+\frac{C_0t}{\sigma-p+1}-1}(\log r)^k
     e^{-\frac{t\theta}{\sigma-p+1}(\log r)^\tau}\,dr.
\end{align*}
Note that
$C_1^\frac{p(\sigma-t)}{\sigma-p+1}\leq(1+C_1)^\frac{p\sigma}{\sigma-p+1}$
and that
\[
R^{\frac{p(\sigma-t)C_1t}{\sigma-p+1}}=R^{\frac{p(\sigma-t)C_1}{\sigma-p+1}\,\frac{1}{\log
R}}=e^{\frac{p(\sigma-t)C_1}{\sigma-p+1}}\leq e^{\frac{p\sigma
C_1}{\sigma-p+1}}.
\]
Thus, with the change of variable $r=e^\xi$, we deduce
\begin{align*}
  I_1&\leq Ct^\frac{p(\sigma-t)}{\sigma-p+1}\int_1^{+\infty}r^{\frac{t}{\sigma-p+1}\pa{C_0-p\sigma C_1+pC_1t+p}}(\log r)^k
        e^{-\frac{t\theta}{\sigma-p+1}(\log r)^\tau}r^{-1}\,dr\\
     &= Ct^\frac{p(\sigma-t)}{\sigma-p+1}\int_0^{+\infty}e^{\frac{t}{\sigma-p+1}\pa{C_0-p\sigma C_1+pC_1t+p}\xi}\xi^k
        e^{-\frac{t\theta}{\sigma-p+1}\xi^\tau}\,d\xi.
\end{align*}
Now recall that by our choice of $C_1$, for $t>0$ small enough, we
have $\pa{C_0-p\sigma C_1+pC_1t+p}<0$. Hence, setting
$\rho=\pa{\frac{t\theta}{\sigma-p+1}}^\frac{1}{\tau}\xi$, we have
\begin{align}
  \label{3}I_1\leq Ct^\frac{p(\sigma-t)}{\sigma-p+1}\int_0^{+\infty}\xi^ke^{-\frac{t\theta}{\sigma-p+1}\xi^\tau}\,d\xi
  =Ct^\frac{p(\sigma-t)}{\sigma-p+1}\pa{\frac{t\theta}{\sigma-p+1}}^{-\frac{k+1}{\tau}}\int_0^{+\infty}\rho^ke^{-\rho^\tau}\,d\rho
  \leq Ct^{\frac{p(\sigma-t)}{\sigma-p+1}-\frac{k+1}{\tau}}.
\end{align}
From \eqref{2}, \eqref{1} and \eqref{3} we conclude that for every
$n\in\enne$ and every small enough $t=\frac{1}{\log R}>0$ we have
\[
 \int_{B_R}{V u^{\sigma-t}}\,d\mu\leq\int_M{V u^{\sigma-t}\varphi_n^s}\,d\mu\leq
    Ct^{-\frac{(p-1)\sigma}{\sigma-p+1}}\sq{t^{\frac{p(\sigma-t)}{\sigma-p+1}-\frac{k+1}{\tau}}+n^{-\frac{t}{\sigma-p+1}}\pa{\log
    (2nR)}^k}
\]
for some fixed positive constant $C$. Passing to the limit as
$n\rightarrow+\infty$ in the previous relation yields
\begin{equation}\label{4}
  \int_{B_R}{V u^{\sigma-t}}\,d\mu\leq
  Ct^{-\frac{(p-1)\sigma}{\sigma-p+1}+\frac{p(\sigma-t)}{\sigma-p+1}-\frac{k+1}{\tau}}.
\end{equation}
Now note that by our assumptions on $\tau,k$ we have
\[
-\frac{(p-1)\sigma}{\sigma-p+1}+\frac{p(\sigma-t)}{\sigma-p+1}-\frac{k+1}{\tau}=
\frac{\sigma}{\sigma-p+1}-\frac{k+1}{\tau}-\frac{pt}{\sigma-p+1}\geq\frac{1}{2}\pa{\frac{\sigma}{\sigma-p+1}-\frac{k+1}{\tau}}:=\delta_*>0
\]
for every small enough $t=\frac{1}{\log R}>0$. Thus \eqref{4} yields
\begin{equation}\label{5}
  \int_{B_{e^{1/t}}}{V u^{\sigma-t}}\,d\mu\leq  Ct^{\delta_*}
\end{equation}
for every small enough $t>0$. Passing to the $\liminf$ as $t$ tends
to $0^+$ in \eqref{5}, we conclude by an application of Fatou's
Lemma that
\[
\int_{M}{V u^{\sigma}}\,d\mu=0,
\]
so that $u\equiv0$ on $M$. \end{proof}

\section{A problem with lower order
terms}\label{RigSec}

In this subsection we consider the semilinear equation
\begin{equation}\label{EQ_Rig1}
    \frac{1}{a(x)}\diver\pa{a(x)\nabla u}+b(x)u+V(x)u^\sigma \leq 0 \quad \text{on }\,  M.
  \end{equation}
We start with the following lemma.
\begin{lemma}\label{lemmaRig}
  Let $u \in W^{1,2}_{\text{loc}}(M)\cap L^\sigma_{\text{loc}}\pa{M, V\, d\mu_0}$ be a nonnegative weak solution
  of \eqref{EQ_Rig1}, with $a$ satisfying \eqref{EQ_propr_a}, $\sigma>1$, $V>0$ a.e. on $M$,  $V\in L^1_{\text{loc}}(M)$ and $b\in L^\frac{2m}{m+2}_{\text{loc}}(M)$.
  Assume there exists a weak solution $z>0$, $z \in
  \operatorname{Lip}_{\text{loc}}(M)$ of
    \begin{equation}\label{EQ_Rig2}
    \frac{1}{a(x)}\diver\pa{a(x)\nabla z}+b(x)z \geq 0 \quad \text{on }\, M.
  \end{equation}
Then $w:=\frac{u}{z}\in W^{1,2}_{\text{loc}}(M)\cap
L^\sigma_{\text{loc}}\pa{M, V\, d\mu_0}$ is a nonnegative weak
solution of
\begin{equation}\label{EQ_Rig3}
    \frac{1}{a(x)z^2(x)}\diver\pa{a(x)z^2(x)\nabla w}+V(x)z^{\sigma-1}(x)w^\sigma \leq 0 \quad \text{on }\, M.
  \end{equation}
\end{lemma}
\begin{proof}
By our assumptions, for every $\varphi\in W^{1,2}(M)\cap
L^\infty(M)$ with compact support and $\varphi\geq0$ a.e. on $M$ we
have
\begin{align}
 \label{16} & -\int_M\left\langle\nabla u,\nabla\varphi\right\rangle\,d\mu+\int_Mbu\varphi\,d\mu+\int_MVu^\sigma\varphi\,d\mu\leq0,\\
 \label{17} & -\int_M\left\langle\nabla z,\nabla\varphi\right\rangle\,d\mu+\int_Mbz\varphi\,d\mu\geq0.
\end{align}
We explicitly note that, by our assumptions, all the integrals in
\eqref{16} and \eqref{17} are finite. Moreover, by a density
argument, we easily see that inequality \eqref{17} also holds for
every $\varphi\in W^{1,2}(M)$ with compact support and
$\varphi\geq0$ a.e. on $M$, not necessarily bounded.

Now we fix $\psi\in W^{1,2}(M)\cap L^\infty(M)$ with compact support
and $\psi\geq0$ a.e. on $M$, and use $\varphi=z\psi \in
W^{1,2}(M)\cap L^\infty(M)$ as a test function in \eqref{16} and
$\varphi=u\psi \in W^{1,2}(M)$ as a test function in \eqref{17}.
Subtracting the resulting inequalities one finds
\begin{equation}\label{18}
 -\int_M\left\langle\nabla u,\nabla\psi\right\rangle z\,d\mu+\int_M\left\langle\nabla
  z,\nabla\psi\right\rangle u\,d\mu+\int_MVu^\sigma\psi z\,d\mu\leq0.
\end{equation}
Since $w=\frac{u}{z}\in W^{1,2}_{\text{loc}}(M)\cap
L^\sigma_{\text{loc}}\pa{M, V\, d\mu_0}$ with
\[
 \nabla w=\frac{1}{z}\nabla u-\frac{u}{z^2}\nabla z\qquad\text{ a.e. on }M,
\]
inequality \eqref{18} becomes
\[
 -\int_M\left\langle\nabla w,\nabla\psi\right\rangle az^2\,d\mu_0+\int_M\pa{V z^{\sigma-1}w^\sigma\psi}az^2\,d\mu_0\leq0.
\]
Then, see also Remark \ref{rem1}, $w$ is a nonnegative weak solution
of \eqref{EQ_Rig3}.
\end{proof}

Combining Lemma \ref{lemmaRig} with Theorem \ref{thm_intro3}, one
can easily obtain the following nonexistence results for nontrivial
nonnegative weak solutions of equation \eqref{EQ_Rig1}.
\begin{proposition}\label{prop1}
Assume there exists a weak solution $z>0$, $z \in
\operatorname{Lip}_{\text{loc}}(M)$ of equation \eqref{EQ_Rig2}
and let $a$ satisfy \eqref{EQ_propr_a}, $\sigma>1$, $V>0$ a.e. on
$M$, $V\in L^1_{\text{loc}}(M)$ and $b\in
L^\frac{2m}{m+2}_{\text{loc}}(M)$. Then any nonnegative weak
solution $u \in W^{1,2}_{\text{loc}}(M)\cap
L^\sigma_{\text{loc}}\pa{M, V\, d\mu_0}$ of \eqref{EQ_Rig1} is
identically null, provided one of the following conditions holds:
\begin{itemize}
    \item[i)] there exist $C_0>0$, $k\in [0, \frac{1}{\sigma-1})$ such that, for every large enough $R>0$ and every
$\eps>0$ sufficiently small,
    \begin{equation*}
      \int_{B_R} V^{-\frac{1}{\sigma-1}+\eps}az^2\,d\mu_0 \leq C R^{\frac{2\sigma}{\sigma-1}+C_0\eps}\pa{\log R}^k,\qquad \text{ or}
    \end{equation*}
    \item[ii)]   there exists $C_0>0$ such that, for every large enough $R>0$ and every $\eps>0$ sufficiently small,
    \begin{align*}
      &\int_{B_R} V^{-\frac{1}{\sigma-1}+\eps}az^2\,d\mu_0 \leq C R^{\frac{2\sigma}{\sigma-1}+C_0\eps}(\log R)^\frac{1}{\sigma-1}
      \qquad\text{ and }\\
      & \int_{B_R} V^{-\frac{1}{\sigma-1}-\eps}az^2\,d\mu_0 \leq C R^{\frac{2\sigma}{\sigma-1}+C_0\eps}(\log
      R)^\frac{1}{\sigma-1},\qquad \text{ or}
    \end{align*}
       \item[iii)] there exist $C_0\geq0$, $k\geq 0$, $\theta>0$, $\tau>\max\{\frac{\sigma-1}{\sigma}\pa{k+1},1\}$ such that, for every large enough $R>0$
       and every $\eps>0$ sufficiently small,
    \begin{equation*}
      \int_{B_{2R}\setminus B_R} V^{-\frac{1}{\sigma-1}+\eps}az^2\,d\mu_0 \leq C R^{\frac{2\sigma}{\sigma-1}+C_0\eps}\pa{\log R}^k e^{-\eps \theta \pa{\log
      R}^\tau}.
    \end{equation*}
  \end{itemize}
\end{proposition}

We now proceed to describe a general setting where one can indeed
produce the desired auxiliary function $z$, in the particular case
when $a\equiv1$ on $M$.

Let us fix a point $o\in M$ and denote by $\textrm{Cut}(o)$ the
{\it cut locus} of $o$. For any $x\in M\setminus
\big[\textrm{Cut}(o)\cup \{o\} \big]$, one can define the {\it
polar coordinates} with respect to $o$, see e.g. \cite{Grig}.
Namely, for any point $x\in M\setminus \big[\textrm{Cut}(o)\cup
\{o\} \big]$ there correspond a polar radius $r(x) := dist(x, o)$
and a polar angle $\theta\in \mathbb S^{m-1}$ such that the
shortest geodesics from $o$ to $x$ starts at $o$ with the
direction $\theta$ in the tangent space $T_oM$. Since we can
identify $T_o M$ with $\mathbb R^m$, $\theta$ can be regarded as a
point of $\mathbb S^{m-1}.$

The Riemannian metric in $M\setminus\big[\textrm{Cut}(o)\cup \{o\}
\big]$ in the polar coordinates reads
\[ds^2 = dr^2+A_{ij}(r, \theta)d\theta^i d\theta^j, \]
where $(\theta^1, \ldots, \theta^{m-1})$ are coordinates in
$\mathbb S^{m-1}$ and $(A_{ij})$ is a positive definite matrix. It
is not difficult to see that the Laplace-Beltrami operator in
polar coordinates has the form
\begin{equation}\label{e70}
\Delta = \frac{\partial^2}{\partial r^2} + \mathcal F(r,
\theta)\frac{\partial}{\partial r}+\Delta_{S_{r}},
\end{equation}
where $\mathcal F(r, \theta):=\frac{\partial}{\partial
r}\big(\log\sqrt{A(r,\theta)}\big)$, $A(r,\theta):=\det
(A_{ij}(r,\theta))$, $\Delta_{S_r}$ is the Laplace-Beltrami
operator on the submanifold $S_{r}:=\partial B(o, r)\setminus
\textrm{Cut}(o)$\,.

$M$ is a {\it manifold with a pole}, if it has a point $o\in M$
with $\textrm{Cut}(o)=\emptyset$. The point $o$ is called {\it
pole} and the polar coordinates $(r,\theta)$ are defined in
$M\setminus\{o\}$.

A manifold with a pole is a {\it spherically symmetric manifold} or
a {\it model}, if the Riemannian metric is given by
\begin{equation}\label{e70b}
ds^2 = dr^2+\psi^2(r)d\theta^2,
\end{equation}
where $d\theta^2$ is the standard metric in $\mathbb S^{m-1}$, and
\begin{equation}\label{26}
\psi\in \mathcal A:=\Big\{f\in C^\infty((0,\infty))\cap
C^1([0,\infty)): f'(0)=1,\, f(0)=0,\, f>0\text{ in }
(0,\infty)\Big\}.
\end{equation}
In this case, we write $M\equiv M_\psi$; furthermore, we have
$\sqrt{A(r,\theta)}=\psi^{m-1}(r)$, so the boundary area of the
geodesic sphere $\partial S_R$ is computed by
\[S(R)=\omega_m\psi^{m-1}(R),\]
$\omega_m$ being the area of the unit sphere in $\mathbb R^m$.
Also, the volume of the ball $B_R(o)$ is given by
\[\mu(B_R(o))=\int_0^R S(\xi)d\xi\,. \]
Moreover we have
\[\Delta = \frac{\partial^2}{\partial r^2}+ (n-1)\frac{\psi'}{\psi}\frac{\partial}{\partial r}+ \frac1{\psi^2}\Delta_{\mathbb S^{m-1}},\]
or equivalently
\[\Delta = \frac{\partial^2}{\partial r^2}+ \frac{S'}{S}\frac{\partial}{\partial r}+ \frac1{\psi^2}\Delta_{\mathbb S^{m-1}},\]
where $\Delta_{\mathbb S^{m-1}}$ is the Laplace-Beltrami operator in
$\mathbb S^{m-1}$.

Observe that for $\psi(r)=r$, $M=\mathbb R^m$, while for
$\psi(r)=\sinh r$, $M$ is the $m-$dimensional hyperbolic space
$\mathbb H^m$.

Let us recall some useful comparison results for sectional and
Ricci curvatures, that will be used in the sequel. For any $x\in
M\setminus\big[\textrm{Cut}(o)\cup\{o\} \big]$, denote by
$Ric_o(x)$ the {\it Ricci curvature} at $x$ in the direction
$\frac{\partial}{\partial r}$. Let $\omega$ denote any pair of
tangent vectors from $T_xM$ having the form
$\left(\frac{\partial}{\partial r} ,X\right)$, where $X$ is a unit
vector orthogonal to $\frac{\partial}{\partial r}$. Denote by
$K_{\omega}(x)$ the {\it sectional curvature} at the point $x\in
M$ of the $2$-section determined by $\omega$. If $M\equiv M_\psi$
is a model manifold, then for any $x=(r, \theta)\in
M\setminus\{o\}$
\[K_{\omega}(x)=-\frac{\psi''(r)}{\psi(r)},\]
and
\[Ric_{o}(x)=-(m-1)\frac{\psi''(r)}{\psi(r)}.\]

Observe moreover that (see \cite{Ichi1}, \cite{Ichi2}, \cite[Section
15]{Grig}), if $M$ is a manifold with a pole $o$ and
\begin{equation}\label{22}
K_{\omega}(x)\leq -\frac{\psi''(r)}{\psi(r)}\quad \textrm{for
all}\;\; x=(r,\theta)\in M\setminus\{o\},
\end{equation}
for some function $\psi\in \mathcal A$, then
\begin{equation}\label{e71}
\mathcal F(r, \theta)\geq (m-1)\frac{\psi'(r)}{\psi(r)}\quad
\textrm{for all}\;\; r>0,\, \theta \in \mathbb S^{m-1}\,.
\end{equation}

On the other hand, if $M$ is a manifold with a pole $o$ and
\begin{equation}\label{24}
Ric_{o}(x)\geq -(m-1)\frac{\psi''(r)}{\psi(r)}\quad \textrm{for
all}\;\; x=(r,\theta)\in M\setminus\{o\},
\end{equation}
for some function $\psi\in \mathcal A$, then
\begin{equation}\label{e71b}
\mathcal F(r, \theta)\leq (m-1)\frac{\psi'(r)}{\psi(r)}\quad
\textrm{for all}\;\; r>0, \theta \in \mathbb S^{m-1}\,.
\end{equation}

We have the following
\begin{lemma}\label{leR}
Let $M$ be a manifold with a pole $o$ and $b\in
L^\frac{2m}{m+2}_{\text{loc}}(M)$. Let
$b_0:\erre^+\rightarrow\erre$ be such that
\begin{equation}\label{23}
b(r,\theta)\geq b_0(r)\quad \textrm{for all}\;\; x=(r,\theta)\in
M\setminus\{o\}.
\end{equation}
Assume that $\psi\in \mathcal{A}$, that
$\zeta:\erre^+\rightarrow\erre$ is a positive weak solution in
$\operatorname{Lip}_{\text{loc}}(\erre^+)$ of
\begin{equation}\label{25}
\pa{\psi^{m-1}\zeta'}'+b_0\psi^{m-1}\zeta\geq0\qquad\textrm{ in
}\erre^+,
\end{equation}
and that either
\begin{itemize}
\item[(A)] $\psi$ satisfies condition \eqref{22} and $\zeta$ is
nondecreasing,
\end{itemize}
or
\begin{itemize}
\item[(B)] $\psi$ satisfies condition \eqref{24} and $\zeta$ is
nonincreasing.
\end{itemize}
Then $z(x):=\zeta(r(x))\in\operatorname{Lip}_{\text{loc}}(M)$ is a
positive weak solution of \eqref{EQ_Rig2}, with $a\equiv1$ on $M$.
\end{lemma}
\begin{proof}
In case condition $(A)$ holds, the result is an easy consequence of
 \eqref{e70}, \eqref{e71}, the monotonicity of $\zeta$ and
condition \eqref{23}. Similarly, when condition $(B)$ holds, the
result follows immediately as before, using \eqref{e71b} in place of
\eqref{e71}.
\end{proof}
We refer the interested reader to the stimulating paper of
Bianchini, Mari, Rigoli \cite{BMR} for results concerning the
existence of a positive solution of \eqref{25} and its precise
asymptotic behavior as $r$ tends to $+\infty$. These combined with
Lemma \ref{leR} and Proposition \ref{prop1} give somehow explicit
nonexistence results for equation \eqref{EQ_Rig1}.

\section{Counterexamples}\label{seclast}
In this section, we will produce three counterexamples to the
previous nonexistence results, all in the particular case of
equation \eqref{EQ_lapl}. Here we follow a similar approach as one
finds in \cite{Grig} and \cite{GrigS}. In the sequel,
$\alpha=\frac{2\sigma}{\sigma-1}$ and $\beta=\frac{1}{\sigma-1}$
as in Definition \ref{defi_volgrowth} with $p=2$, while $M$ will
always denote a model manifold with a pole $o$ and metric given by
\eqref{e70}. Set $B_R\equiv B_R(o)$ and $r\equiv r(x)=dist(x,o)$
for any $x\in M$.

Let $spec(-\Delta)$ be the spectrum in $L^2(M)$ of the operator
$-\Delta$. Note that (see \cite[Section 10]{Grig})
\[spec(-\Delta)\subseteq [0,\infty)\,.\]
Denote by $\bar\lambda(M)$ the bottom of $spec(-\Delta)$, that is
\[\bar\lambda(M):=\inf spec(-\Delta)\,.\]
By \cite{Br}, for each fixed $x\in M$, there holds
\begin{equation}\label{e20}
\bar\lambda(M)\leq \frac 1
4\left[\limsup_{R\rightarrow+\infty}\frac{\log\mu(B_R(x))}{R}\right]^2\,.
\end{equation}
For any $\rho>0$, let $\lambda_\rho$ be  the first Dirichlet eigenvalue of the Laplace operator in $B_\rho$, that is the smallest number $\lambda_\rho$ for which the problem
\begin{equation}\label{e21}
\begin{cases}
\Delta u + \lambda u \,=\, 0  &  \,\,  \textrm{in} \  B_\rho  \, , \\
 u\,=\,0  & \ \textrm{on} \  \partial B_\rho  \, . \\
\end{cases}
\end{equation}
has a non-zero solution. Indeed, $\lambda_\rho$ coincides with the
bottom of the spectrum of the operator $-\Delta$ in $L^2(B_\rho)$
with domain $C^\infty_0(B_\rho)$. It is easily checked, see e.g.
****ref uberbook***, that $\lambda_\rho\geq 0;$ moreover,
$\lambda_{\rho_1}\geq \lambda_{\rho_2}$ if $\rho_1<\rho_2$, and
$\lambda_\rho\to \bar\lambda(M)$ as $\rho\to\infty.$

\medskip

In the sequel we shall make use of the following result (see
\cite{GrigS}).
\begin{proposition}\label{propGS}
Let $\sigma>1$, $r_0>0$, $A\in C^1((r_0, \infty))$ with $A>0$ and
$\int_{r_0}^\infty\frac{dr}{A(r)}<\infty.$ Let $B\in C((r_0,
\infty))$ be such that
\[\int_{r_0}^\infty [\gamma(r)]^\sigma |B(r)|dr<\infty,\]
where
\[\gamma(r):=\int_{r}^\infty  \frac{d\xi}{A(\xi)}\qquad\text{ for }r\geq r_0. \]
 Then the equation
 \[\pa{A(r) y'}' + B(r) y^\sigma\,=\,0\qquad\text{ for }r>R_0,\]
for $R_0>r_0$ sufficiently large, admits a positive solution $y(r)$
such that $y(r)\sim \gamma(r)$ as $r\to \infty$.
\end{proposition}

\begin{exe}\label{exe1}
Let $\psi\in \mathcal A$, see \eqref{26}, with
\[\psi(r):=\begin{cases}
r  &  \,\,  \textrm{if} \  0\leq r<1  \, , \\
 [r^{\alpha-1}(\log r)^{\beta_0} ]^{\frac 1{m-1}} & \ \textrm{if} \  r > 2 \,; \\
\end{cases}
\]
here $\beta_0>\beta$. Let $0<\delta < \beta_0 -\beta$ and define
\[V(x)\equiv V(r):= (\log(2+r))^{\frac{\delta}{\beta}} \quad \textrm{for all}\;\; x\in M\,.\]
For any $R>0$ sufficiently large we have
\[S(R)= \omega_m R^{\alpha-1}(\log R)^{\beta_0},\quad \mu(B_R)\simeq C R^\alpha (\log R)^{\beta_0}\,;\]
thus, thanks to \eqref{e20}, we have $\bar\lambda(M)=0.$ Moreover,
there holds
\begin{equation}\label{e1}
\int_{B_R} V^{-\beta}(x)\, d\mu \geq C R^\alpha (\log R)^{\beta_0
-\delta},
\end{equation}
with $\beta_0-\delta>\beta$. Hence, in view of \eqref{e1}, neither
condition \eqref{EQ_hp1} nor condition \eqref{EQ_hp2} is satisfied.
Furthermore, observe that \eqref{e32} holds true, while \eqref{e33}
fails. This is essentially due to the choice of $\psi$.

Note that for any $r_0>0$,
\begin{equation}\label{e3}
\int_{r_0}^\infty\frac{d\xi}{S(\xi)}<\infty\,;
\end{equation}
moreover, for $r> 0$ sufficiently large,
\[\gamma(r) := \int_r^\infty \frac{d\xi}{S(\xi)} \simeq \frac C{r^{\alpha-2}(\log r)^{\beta_0}} \,.\]
Hence for $r_0>0$ large enough
\begin{align}
\label{e2}\int_{r_0}^\infty[\gamma(r)]^\sigma V(r) S(r)\, dr &\leq C
\int_{r_0}^\infty \frac{[\log(2+r)]^{\delta/\beta}r^{\alpha-1}
(\log r)^{\beta_0}}{r^{\sigma(\alpha-2)}(\log r)^{\beta_0 \sigma}}\,dr\\
\nonumber&\leq C \int_{r_0}^\infty (\log
r)^{\beta_0(1-\sigma)+\frac{\delta}{\beta}}\,\frac{dr}{r}<\infty\,.
\end{align}
In view of \eqref{e3}-\eqref{e2}, from Proposition \ref{propGS} with
$A(r)=S(r)$ and $B(r)=S(r)V(r)$, we have that there exists a
solution $y=y(r)$ of
\begin{equation}\label{e7}
y''(r) + \frac{S'(r)}{S(r)} y'(r) + V(r) [y(r)]^\sigma = 0 ,\quad r>R_0,
 \end{equation}
for some $R_0>r_0$. Furthermore, $y(r)>0$ for all $r\in [R_0,
\infty)$ and $y(r) \sim \gamma(r)$ as $r\to \infty$.

Now for any $\rho>0$, let $v_\rho$ be the solution to the eigenvalue
problem \eqref{e21} with $\lambda=\lambda_\rho$, which we can assume
is normalized by setting $v_\rho(o)=1$. Thus we have that
$v_\rho\equiv v_\rho(r)$, that $0<v_\rho(r)\leq1$ and that
$v_\rho(r)$ is decreasing for $r\in[0,\rho]$. For any $\rho
>R_0$, define $\displaystyle m:=\inf_{[R_0,
\rho)}\frac{y(r)}{v_\rho(r)}$ and for any fixed $\xi\in(R_0,\rho)$
let
\[\tilde u (x):= \begin{cases}
m v_\rho(r) &  \,\,  \textrm{in} \  B_\xi, \\
y(r) & \ \textrm{on} \  \partial B_\xi^c. \\
\end{cases}
\]
Since $\bar\lambda(M)=0$, as in \cite{GrigS} one can prove that for
some $\xi\in (R_0, \rho)$ we have $\tilde u \in C^1(M)$, and thus
$\tilde u\in W^1_{\text{loc}}(M)$. Moreover
\begin{align*}
\Delta (m v_\rho) + \frac{\lambda_\rho}{m^{\sigma-1}}(m
v_\rho)^\sigma&= 0 \qquad \textrm{ in } B_\rho,\\
\Delta y + V y^\sigma &= 0 \qquad \textrm{ in } B_{R_0}^c.
\end{align*}
Define $M_\rho=\max_{\overline B_\rho}V$ and
 \[\delta=\min\set{1,m^{-1}\lambda_\rho^\frac{1}{\sigma-1}M_\rho^{-\frac{1}{\sigma-1}}};\]
then, since $V>0$ on $M$, $\delta>0$ and on $B_\rho$ we have
\begin{align*}
\Delta (\delta m v_\rho) + V(\delta m v_\rho)^\sigma&\leq \Delta
(\delta m v_\rho) + M_\rho(\delta m v_\rho)^\sigma\\
 &\leq\Delta (\delta m v_\rho) + \frac{\lambda_\rho}{(\delta m)^{\sigma-1}}(\delta m
 v_\rho)^\sigma=0.
\end{align*}
On the other hand on $B_{R_0}^c$ we have
\begin{align*}
\Delta (\delta y) + V(\delta y)^\sigma&\leq \delta\Delta y + \delta
Vy^\sigma=0.
\end{align*}
Thus we see that the function $u=\delta \tilde u$ is positive and
satisfies
\[ \Delta u+Vu^\sigma\leq0\qquad \text{ on }M.\]
\end{exe}

\begin{exe}
Let $\psi\in \mathcal A$ with
\[\psi(r):=\begin{cases}
r  &  \,\,  \textrm{if} \  0\leq r<1  \, , \\
 [r^{\alpha-1}(\log r)^{\beta} ]^{\frac 1{m-1}} & \ \textrm{if} \  r > 2 \, . \\
\end{cases}
\]
Let $\delta>0$ and define
\[V(x)\equiv V(r):= (\log(2+r))^{-\frac{\delta}{\beta}} \quad \textrm{for all}\;\; x\in M\,.\]
For any $R>0$ sufficiently large we have
\[S(R)= \omega_m R^{\alpha-1}(\log R)^{\beta},\quad \mu(B_R)\simeq C R^\alpha (\log R)^{\beta}\,,\]
thus, thanks to \eqref{e20}, we conclude taht $\lambda_1(M)=0$.
Moreover, there holds
\begin{equation}\label{e4}
\int_{B_R} V^{-\beta}(x)\, d\mu \geq C R^\alpha (\log R)^{\beta
+\delta}\,.
\end{equation}
Observe that in view of \eqref{e4}, neither condition \eqref{EQ_hp1}
nor condition \eqref{EQ_hp2} is satisfied. Moreover, note that
\eqref{e32} holds, while \eqref{e33} fails. This is essentially due
to the choice of $V$.

For any $r_0>0$ we have
\begin{equation}\label{e5}
\int_{r_0}^\infty\frac{d\xi}{S(\xi)}<\infty\,;
\end{equation}
moreover, for $r> 0$ large enough,
\[\gamma(r) := \int_r^\infty \frac{d\xi}{S(\xi)} \simeq \frac C{r^{\alpha-2}(\log r)^{\beta}} \,.\]
Hence for $r_0>0$ large enough
\begin{align}
\label{e6}\int_{r_0}^\infty[\gamma(r)]^\sigma V(r) S(r)\, dr &\leq C \int_{r_0}^\infty \frac{[\log(2+r)]^{-\delta/\beta}r^{\alpha-1}(\log r)^{\beta}}{r^{\sigma(\alpha-2)}(\log r)^{\beta \sigma}}\,dr\\
\nonumber&\leq C \int_{r_0}^\infty \frac 1{(\log
r)^{\beta(\sigma-1)+\frac{\delta}{\beta}}}\, \frac{dr}{r}<\infty.
\end{align}
In view of \eqref{e5}-\eqref{e6}, from Proposition \ref{propGS} with
$A(r)=S(r)$ and $B(r)=S(r)V(r)$, we have that there exists a
solution $y=y(r)$ of \eqref{e7}, for some $R_0>0$. Furthermore,
$y(r)>0$ for all $r\in [R_0, \infty)$ and $y(r) \sim \gamma(r)$ as
$r\to \infty$. Since $\lambda_1(M)=0$ and $V>0$ on $M$, by the same
arguments as in the previous Example \ref{exe1}, we can construct
$u\in C^1(M)$, with $u=u(r)>0$ on $M$, which satisfies
\[\Delta u + V
u^\sigma \leq 0 \quad \textrm{ in } M.\]
\end{exe}

\begin{exe}
Let $\psi\in \mathcal A$ with
\[\psi(r):=\begin{cases}
r  &  \,\,  \textrm{if} \  0\leq r<1  \, , \\
e^{\sqrt r} & \ \textrm{if} \  r > 2 \, . \\
\end{cases}
\]
For any  sufficiently large $R>0$  we have
\[S(R)= \omega_m  e^{(m-1)\sqrt r}, \quad \mu(B_R) \simeq C e^{(m-1)\sqrt R}[(m-1)\sqrt R -1],\]
for some $C>0$. Note that $\bar \lambda(M)=0$ by \eqref{e20},
since
$$\limsup_{R\to\infty}\frac{\log (\mu(B_R))}{R}=0.$$
Let
\[\eta = \frac{m-1}{\beta}=(m-1)(\sigma-1),\quad
\theta=\frac{\sigma+1}{\sigma-1}\] and define
\[V(x)\equiv V(r):= e^{\eta \sqrt r}(1+r)^{-\frac{\theta}{\beta}} \quad \textrm{ for all } x\in M.\]
Then for $\eps>0$ small enough and $R>0$ sufficiently large we have
\begin{equation}\label{e8}
\int_{B_R} V^{-\beta+\eps}(x)\, d\mu \geq C \int_2^R e^{\eps
\eta\sqrt{r}}(1+r)^{\theta\pa{1-\eps/\beta}}\,dr\geq
Ce^{\eps\eta\sqrt{R}}.
\end{equation}
On the other hand
\begin{equation}\label{e8bis}
\int_{B_R} V^{-\beta-\eps}(x)\, d\mu \leq C \int_0^R e^{-\eps
\eta\sqrt{r}}(1+r)^{\theta\pa{1+\eps/\beta}}\,dr\leq C
R^{\theta+1+\frac{\theta}{\beta}\eps}=CR^{\alpha+\frac{\theta}{\beta}\eps}.
\end{equation}

Observe that, in view of \eqref{e8}, neither condition
\eqref{EQ_hp1} nor the first inequality in condition \eqref{EQ_hp2}
is satisfied. On the other hand, by \eqref{e8bis} the second
inequality in \eqref{EQ_hp2} holds. This is essentially due to the
choice of $V$. Note moreover that only the second inequality in
\eqref{e32} is not satisfied, while the first inequality in
\eqref{e32} and \eqref{e33} hold.

Note that for any $r_0>0$
\begin{equation}\label{e9}
\int_{r_0}^\infty\frac{d\xi}{S(\xi)}<\infty\,;
\end{equation}
moreover, for $r>0$  sufficiently large,
\[\gamma(r) := \int_r^\infty \frac{d\xi}{S(\xi)} \simeq C e^{-(m-1)\sqrt r} \sqrt r.\]
Hence
\begin{equation}\label{e10}
\int_{r_0}^\infty[\gamma(r)]^\sigma V(r) S(r) dr  \leq
C\int_{r_0}^\infty e^{[-(m-1)(\sigma-1)+\eta]\sqrt
r}r^{\frac{\sigma}2-\frac{\theta}{\beta}}<\infty.
\end{equation}
In view of \eqref{e9}-\eqref{e10}, from Proposition \ref{propGS}
with $A(r)=S(r)$ and $B(r)=S(r)V(r)$, we have that there exists a
solution $y=y(r)$ of \eqref{e7}, for some $R_0>0$. Furthermore,
$y(r)>0$ for all $r\in [R_0, \infty)$ and $y(r) \sim \gamma(r)$ as
$r\to \infty$. Since $\bar \lambda(M)=0$ and $V>0$ on $M$, by the
same arguments as in Example \ref{exe1} above, we can construct
$u\in C^1(M)$ with $u=u(r)>0$ on $M$, which is a weak solution of
\[\Delta u + V u^\sigma \leq 0 \quad
\textrm{ in } M.\]
\end{exe}

\bibliographystyle{plain}

\bibliography{BiblioMMPFin}

\def\cprime{$'$}
\begin{thebibliography}{10}

\bibitem{BMR}
B.~Bianchini, L.~Mari, and M.~Rigoli.
\newblock Yamabe type equations with sign-changing nonlinearities on
  non-compact {R}iemannian manifolds.
\newblock {\em preprint}.

\bibitem{Br}
R.~Brooks.
\newblock A relation between growth and the spectrum of the {L}aplacian.
\newblock {\em Math. Z.}, 178(4):501--508, 1981.

\bibitem{DaLu}
L.~D'Ambrosio and S.~Lucente.
\newblock Nonlinear {L}iouville theorems for {G}rushin and {T}ricomi operators.
\newblock {\em J. Differential Equations}, 193(2):511--541, 2003.

\bibitem{Gidas}
B.~Gidas.
\newblock Symmetry properties and isolated singularities of positive solutions
  of nonlinear elliptic equations.
\newblock In {\em Nonlinear partial differential equations in engineering and
  applied science ({P}roc. {C}onf., {U}niv. {R}hode {I}sland, {K}ingston,
  {R}.{I}., 1979)}, volume~54 of {\em Lecture Notes in Pure and Appl. Math.},
  pages 255--273. Dekker, New York, 1980.

\bibitem{GidasSpruck}
B.~Gidas and J.~Spruck.
\newblock Global and local behavior of positive solutions of nonlinear elliptic
  equations.
\newblock {\em Comm. Pure Appl. Math.}, 34(4):525--598, 1981.

\bibitem{GilTru}
D.~Gilbarg and N.~S. Trudinger.
\newblock {\em Elliptic partial differential equations of second order}.
\newblock Classics in Mathematics. Springer-Verlag, Berlin, 2001.
\newblock Reprint of the 1998 edition.

\bibitem{Grig}
A.~Grigor{\cprime}yan.
\newblock Analytic and geometric background of recurrence and non-explosion of
  the {B}rownian motion on {R}iemannian manifolds.
\newblock {\em Bull. Amer. Math. Soc. (N.S.)}, 36(2):135--249, 1999.

\bibitem{GrigKond}
A.~Grigor{\cprime}yan and V.~A. Kondratiev.
\newblock On the existence of positive solutions of semilinear elliptic
  inequalities on {R}iemannian manifolds.
\newblock In {\em Around the research of {V}ladimir {M}az'ya. {II}}, volume~12
  of {\em Int. Math. Ser. (N. Y.)}, pages 203--218. Springer, New York, 2010.

\bibitem{GrigS}
A.~Grigor{\cprime}yan and Y.~Sun.
\newblock On non-negative solutions of the inequality {$\Delta u + u^\sigma
  \leq 0$} on {R}iemannian manifolds.
\newblock {\em Comm. Pure Appl. Math.}, To appear.

\bibitem{Ichi1}
K.~Ichihara.
\newblock Curvature, geodesics and the {B}rownian motion on a {R}iemannian
  manifold. {I}. {R}ecurrence properties.
\newblock {\em Nagoya Math. J.}, 87:101--114, 1982.

\bibitem{Ichi2}
K.~Ichihara.
\newblock Curvature, geodesics and the {B}rownian motion on a {R}iemannian
  manifold. {II}. {E}xplosion properties.
\newblock {\em Nagoya Math. J.}, 87:115--125, 1982.

\bibitem{Kurta}
V.~V. Kurta.
\newblock On the absence of positive solutions to semilinear elliptic
  equations.
\newblock {\em Tr. Mat. Inst. Steklova}, 227(Issled. po Teor. Differ. Funkts.
  Mnogikh Perem. i ee Prilozh. 18):162--169, 1999.

\bibitem{MitPohoz359}
E.~Mitidieri and S.~I. Pohozaev.
\newblock Absence of global positive solutions of quasilinear elliptic
  inequalities.
\newblock {\em Dokl. Akad. Nauk}, 359(4):456--460, 1998.

\bibitem{MitPohoz227}
E.~Mitidieri and S.~I. Pohozaev.
\newblock Absence of positive solutions for quasilinear elliptic problems in
  {${\bf R}^N$}.
\newblock {\em Tr. Mat. Inst. Steklova}, 227(Issled. po Teor. Differ. Funkts.
  Mnogikh Perem. i ee Prilozh. 18):192--222, 1999.

\bibitem{MitPohozAbsence}
E.~Mitidieri and S.~I. Pohozaev.
\newblock A priori estimates and the absence of solutions of nonlinear partial
  differential equations and inequalities.
\newblock {\em Tr. Mat. Inst. Steklova}, 234:1--384, 2001.

\bibitem{MitPohozMilan}
E.~Mitidieri and S.~I. Pohozaev.
\newblock Towards a unified approach to nonexistence of solutions for a class
  of differential inequalities.
\newblock {\em Milan J. Math.}, 72:129--162, 2004.

\bibitem{Mont}
D.~D. Monticelli.
\newblock Maximum principles and the method of moving planes for a class of
  degenerate elliptic linear operators.
\newblock {\em J. Eur. Math. Soc. (JEMS)}, 12(3):611--654, 2010.

\bibitem{PoTe}
S.~I. Pohozaev and A.~Tesei.
\newblock Nonexistence of local solutions to semilinear partial differential
  inequalities.
\newblock {\em Ann. Inst. H. Poincar\'e Anal. Non Lin\'eaire}, 21(4):487--502,
  2004.

\bibitem{P1}
F.~Punzo.
\newblock Blow-up of solutions to semilinear parabolic equations on
  {R}iemannian manifolds with negative sectional curvature.
\newblock {\em J. Math. Anal. Appl.}, 387(2):815--827, 2012.

\bibitem{PuTe}
F.~Punzo and A.~Tesei.
\newblock On a semilinear parabolic equation with inverse-square potential.
\newblock {\em Atti Accad. Naz. Lincei Cl. Sci. Fis. Mat. Natur. Rend. Lincei
  (9) Mat. Appl.}, 21(4):359--396, 2010.

\end{thebibliography}
\end{document}